\documentclass[a4paper,12pt]{amsart}
\usepackage{amsmath, amsfonts, amssymb, amsthm, euscript, amscd, latexsym, mathrsfs, bm, color, bbm, xcolor,mathtools, varwidth, multicol, xfrac, nicefrac,graphicx}
\usepackage{appendix}
\usepackage[pdftex,colorlinks=true,
                       pdfstartview=FitV,
                       linkcolor=blue,
                       citecolor=blue,
                       urlcolor=blue,
           ]{hyperref}
\usepackage{colortbl}

\usepackage{evelina-macros}

\usepackage{todonotes}

\newcommand\del[1]{}

\newcommand\dela[1]{}

\numberwithin{equation}{section}

\title{Singular SPDE{\scriptsize s} with the Cauchy-Riemann operator on a torus}

\begin{document}

\author{Zdzis\l aw Brze\'zniak}

\address{Department of Mathematics, University of York, York, YO10 5DD, UK}
\email{zdzislaw.brzezniak@york.ac.uk}

\author{Mikhail Neklyudov}

\address{Departamento de Matem\'atica, Universidade Federal do Amazonas, Manaus, Brazil}
 \email{misha.neklyudov@gmail.com}

\author{Evelina Shamarova}

\address{Departamento de Matem\'atica, Universidade Federal da Para\'iba, Jo\~ao Pessoa, Brazil}
\email{evelina.shamarova@academico.ufpb.br}

%\subjclass[2020]{60H15, 37K99}

%\keywords{Cauchy-Riemann equations}

\begin{abstract}
We prove the existence of solution to the
following $\mathbb{C}^3$-valued singular SPDE on 
the 2D torus $\mathbb{T}^2$:
\begin{align}
\label{CR}
\partial_{\bar z} r = r \times \overline{r} + i \, \gamma \, {\mathscr W},
\tag{CR}
\end{align}
where $\partial_{\bar z}: = \frac12(\partial_x + i \partial_y)$ is the Cauchy-Riemann operator
on $\mathbb{T}^2$,
 ${\mathscr W} = ({\scriptstyle {\mathscr W}_1}, {\scriptstyle {\mathscr W}_2}, {\scriptstyle {\mathscr W}_3})$ is a real 3D white noise on $\mathbb{T}^2$ whose
component ${\scriptstyle {\mathscr W}_3}$ has zero mean over $\mathbb{T}^2$, $\gamma: = (\gamma_1,\gamma_2,\gamma_3)$
is an $\mathbb{R}^3$-vector and $\gamma \, {\mathscr W}: = (\gamma_1 {\scriptstyle {\mathscr W}_1}, \gamma_2 {\scriptstyle {\mathscr W}_2}, \gamma_3 {\scriptstyle {\mathscr W}_3})$.
\end{abstract}

\maketitle

\vspace{5mm}

{\scriptsize

\vspace{1mm}

}

\section{Introduction}
\subsection{Motivation}
We start by noticing that
the Cauchy-Riemann operator $\pl_{\bar z}: = \frac12(\pl_x + i \pl_y)$ and its complex conjugate
$\pl_z: = \frac12(\pl_x - i \pl_y)$ are building blocks of  the 2D free Dirac operator 
\aaa{
\lb{dirac}
\mc D: = -2i \begin{pmatrix}
0 & \pl_z\\
\pl_{\bar z} & 0
\end{pmatrix} = -i\sg_1 \pl_x - i\sg_2 \pl_y,
}
where $\sg_1: =\begin{pmatrix}
0 & 1\\
1 & 0
\end{pmatrix}$ and $\sg_2: =\begin{pmatrix}
0 & -i\\
i & 0
\end{pmatrix}$ 
are the Pauli matrices, see, e.g. \cite{Behrndt2023}. Since equation \eqref{CR} can be rewritten as  $\pl_z \ovl r =  - r \x \ovl r - i\gm \fv$,   by   using the free Dirac operator \rf{dirac}, we see that equation \rf{CR} can be written in the following  equivalent form 
\aaa{
\lb{dirac-1111}
\mc D \vec{\ovl r}{r} =  \vec{2i}{-2i} r\x\ovl r + \vec{2}{-2}(\gm\fv).
}
Equation \eqref{dirac-1111} can be regarded as a white-noise perturbation of a stationary Dirac equation. Dirac operators have wide-ranging applications in quantum mechanics, relativistic field theory, differential geometry, 
and Clifford algebras \cite{GilbertMurray1991, Thaller1992}. By rewriting equation  \eqref{CR} in terms of  Dirac operators, we establish a connection to these areas, suggesting potential applications of our equation to these fields. We also remark that  stationary deterministic Dirac equations have been studied in the literature \cite{Chuprikov-2015, esteban-sere-1995}.

The quadratic term $r\x \ovl r$ in equation \eqref{dirac-1111} represents a nonlinear self-interaction, where the solution $r$
couples with its own complex conjugate,  introducing an interesting mathematical structure.
 One possible interpretation of this cross-product term is that it describes a nonlinear coupling between spinor-like components in a quantum mechanical system.
 Nonlinearities of quadratic type appear, for example, in Dirac-Klein-Gordon systems containing an interaction term that couples a Dirac field with a scalar field, as well as an interaction term representing a self-interaction of the Dirac field \cite{Selberg}. Quadratic nonlinearities  appear also in quadratic Dirac equations, which describe self-interacting fermions \cite{Selberg}.
From a mathematical perspective, the cross-product term plays a crucial role in balancing singularities introduced by the white noise $\fv$.
Specifically, this nonlinearity contains a self-correcting 
mechanism that cancels the most singular terms, 
ensuring the existence of solutions as 
limits in a H\"older space of negative order.
This interplay between the structure of our nonlinearity 
and the solvability of the equation highlights an interesting and novel aspect 
of this work.

Another motivation for studying equation \eqref{CR} comes from the  
following stochastic equation  
\aaa{
  \lb{B-L}
 \lap B=2 \pl_x B\times \pl_y B +  4\gm \fv,
 %\tag{LLG}
 }
 where $B:\Tor \to \Rnu^3$. The homogeneous version of 
 equation \rf{B-L} can be derived
 from the homogeneous Landau-Lifshitz-Gilbert (LLG) equation 
 \[ u\x \lap u = 0 \mbox{ with } |u|=1
 \]
 via  the  B\"acklund transformation 
$\nab B: = \big(u\x \pl_y u, -u \x \pl_x u\big)$,
see e.g. \cite[p. 50-52]{Helein2002}.
 This transformation
 provides a mathematical connection between the LLG equation and 
 equation \eqref{B-L}, suggesting that \eqref{B-L} can be viewed as a stochastic generalization of the LLG equation.

It is straightforward to prove, on a heuristic level,  that $r= i\pl_z B$ 
is a solution  to \rf{CR}  if and only if $B$ is a solution to \rf{B-L}.
This provides an additional motivation for studying equation \eqref{CR}, 
as it establishes a link between our stochastic equation 
and a well-known physical model.
 In Section \ref{c-mean-1111}, we deduce the existence of 
 a solution $B$ to \rf{B-L} from the existence of a  solution $r$ to \rf{CR} 
 under an additional assumption on the white noise $\fv$.

Thus, in this work we are motivated to study equation \rf{CR}. 
To the best of our knowledge, the Cauchy-Riemann operators
in the context of singular SPDEs have not been studied before. 
Moreover, the periodic setting introduces
additional challenges, as the solvability of \rf{CR} depends 
on the mean-zero condition of the right-hand side. 
More specifically, as we demonstrate in Lemma \ref{Lav1111},
the equation $\pl_{\bar z} r = f$ is solvable 
if and only if $\int_\Tor f = 0$.

We remark that, unlike most singular SPDEs considered in the literature, 
see e.g. \cite{bcz2023-2, bfm2016, DaPrato+Debussche_2002, FH20, hairer}, 
ours is of the first-order, 
where the kernel of the operator $\pl_{\bar z}$
increases the regularity by  $1$. In this sense,
and also since $\pl_{\bar z} \pl_z = \frac12 \lap$, 
our work is 
related to recent studies of singular SPDEs involving the fractional 
Laplacian $(-\lap)^\frac12$ 
\cite{Berglund2018, ChiariniLandim2022, Duch2025, DuchGubinelliRinaldi2023},
which can also be regarded as a first-order differential operator 
due to its kernel's regularity-increasing properties.  However, our focus on the Cauchy-Riemann operator introduces different challenges and requires new techniques, distinct from those used for fractional Laplacians. By developing these techniques, 
 we aim to extend the existing framework to a broader class of first-order singular SPDEs.

\subsection{The method}
{In this paper, we consider the  mild formulation of equation  \rf{CR}, see e.g. \eqref{sys_r1111},  
with the complex-valued Green function $G=2\pl_z K$ 
of  the operator $\pl_{\bar z}$ on
$\Tor$, where $K$ is the Green function for the Laplacian with zero mean on the torus $\Tor$. 
From this ``mild" equation,   we deduce  that the expected regularity of a solution to equation  \rf{CR} is $0^-$. This implies that the quadratic term $r\x \ovl r$ is ill-defined.
To overcome  this difficulty,
we apply the method used by  Da Prato--Debussche in \cite{DaPrato+Debussche_2002} which involves
decomposing the solution $r$ into a singular part and a smoother remainder. 
More specifically,
we rewrite the mild formulation of  equation \rf{CR} with respect to
a new  variable $R = r+ \gm \xi$, where $\xi: = 2i\, G \ast \fv$ is the singular term.
This decomposition allows us to isolate and control 
the additive term in the  equation for $R$, namely, the term
\aaa{
\lb{terms}
\zeta : = 2  G\ast (\xi \x \ovl{\xi}).
}
Let us observe that the term $\tet := \xi \x \ovl{\xi}$ is ill-defined, since $\xi$ has negative regularity $0^-$.
We overcome this difficulty by employing a  smooth approximation of $\xi$,  redefining 
$\tet$ as the corresponding limit, and 
showing that  its components, we take the first component for certainty,  are of the following form:
\aaa{
\lb{noise-1111}
 \tet_1 = 2\pl_1\big[(K\ast \ww) \Re \xi_3 \big] -
 2\pl_2 \big[(K\ast \ww) \Im \xi_3 \big].
 }
 Deducing the above expression involves a renormalization procedure 
 via  cancellation of divergent terms.
  We observe that \rf{noise-1111} is well defined 
 since the expression inside the square brackets  
 can be understood in the Young product sense, see  \cite{bcz2023-1}. 
Moreover,  this expression    suggests that the expected regularity 
 of $\tet$ is $-1^{-}$. Had this been the case,  the expected regularity of 
 $\zeta$, in \rf{terms}, would be  $0^{-}$ and  therefore  the expected  regularity  of  $R$ 
 would also be $0^{-}$.  Such a regularity is  too low for the product $R\x \ovl{\xi}$ to be well-defined. Hence,  a fixed-point argument  cannot be realized in a H\"older space of a positive regularity. A natural
 way to remedy this  would be to consider a lift of our equation to 
 a space of modeled distributions, as in  \cite{bcz2023-1, bcz2023-2, FH20, hairer}.
 %a model space of a regularity structure 
 
 However, we are able to  avoid using this approach by proving, exploiting another reformulation of \eqref{noise-1111},  that $\zeta$ has  
 regularity $1^{-}$, see Proposition \ref{prop-2222} and Corollary \ref{cor-2222}.
 This subtle regularity analysis 
 represents a significant technical innovation and 
 is the cornerstone of our existence proof.
 %the main difficulty of our work and 
 It allows us 
to carry out the     Banach   fixed-point argument in the H\"older
 space  of a positive regularity $\C^{1-\kp}(\Tor)$, $\kp\in  (0,\frac12)$, by understanding 
 the vector product $R\x \ovl{\xi}$, and other similar, 
  in the  Young sense.
 The  main ingredient of our regularity analysis is
 the following representation  of the component $\tet_1$, as well as $\tet_2$ and $\tet_3$: 
 \aa{
%\lb{noise-1111}
 \tet_1 = 2\pl_1\big[\big(K\ast \ww - (K\ast \ww)(x)\big) \Re \xi_3 \big] -
 2\pl_2 \big[\big(K\ast \ww - (K\ast \ww)(x)\big) \Im \xi_3 \big],
 }
for an arbitrary $x\in\Tor$.  \dela{We get rid of the auxiliary term 
 $(K\ast \ww)(x)$  by differentiation and renormalization.} As in \rf{noise-1111},
 a renormalization occurs through a cancellation of certain divergent terms.
 We remark that including $(K\ast \ww)(x)$ in the expression \rf{noise-1111} is critical in our analysis. %of the regularity of $\zeta$. 
 % it represents a significant technical innovation and is the cornerstone of our existence proof.

It remains to notice that 
to analyze the ill-defined product $\xi\x \ovl\xi$, we employ a smooth white noise approximation by the standard mollifier and show the following convergence
rate $\|\xi\x \ovl\xi - \xi^\eps\x \ovl{\xi^\eps}\|_{\C^{-\kp}(\Tor)}
\ls \eps^{\frac\kp{8}}\|\fv\|_{\C^{-1-\kp}(\Tor)}$. 
In  \cite{DaPrato+Debussche_2002}, the authors addressed 
ill-defined products on a 2D torus by proving a convergence of Galerkin approximations in a negative-order Sobolev space.
}

\subsection{Main result}

\begin{definition}
\lb{white-noise}
A 3D white noise on $\Tr^d$ is a Gaussian random field 
\[
\fv= (\w,\ww,\www): \Tr^d \times \Omega \to \mathbb{R}^3
\]
with the following Fourier series representation:
\aaa{
\lb{FS-2222}
\v_j  %\sum_{k\in \Znu^d} \eta_{kj} e^{i(k,x)}
= \eta_j + \Re \Big[\sum_{k\in \Znu^d\setminus \{0\}} \eta_{kj} 
e^{i(k,x)}\Big], 
\qquad j=1,2,3,
}
where $\eta_j$, $j=1,2,3$, are standard i.i.d.~real  Gaussian random variables 
and $\eta_{kj}$, $j=1,2,3$, $k\in\Znu^d\setminus \{0\}$,
are standard i.i.d.~complex Gaussian random variables. 
% Put $\eta_{(-k)j} = \ovl{\eta_{kj}}$ for all $j=1,2,3$, 
% $k\in\Znu^d\setminus \{0\}$.
Furthermore, the series \rf{FS-2222} converge in $\mc D'(\Tr^d)$.
% By $\Znu^d_+$ we understand the subset of vectors from $\Znu^d\setminus \{0\}$
% whose first non-zero component is positive.
\end{definition}
\begin{remark}
\lb{zero-mean-1111}
\rm
We say that the white noise $\v_j$ has zero mean over $\Tr^d$
if, in representation \rf{FS-2222}, $\eta_j=0$. We say that 
$\fv$ has zero mean over $\Tr^d$ if $\w,\ww$, and $\www$
have zero means over $\Tr^d$.
\end{remark}

The following is the main result of our paper.

\begin{theorem}
\lb{thm-main-1111}
Let $(\Om,\mc F, \PP)$ be a probability space and
let $\fv = (\w,\ww,\www)$ be a 3D white noise on $\Tor$ such that $\www$
has zero mean over $\Tor$. Further let
$(\w^\eps, \ww^\eps, \www^{\eps})$ be a mollification of $\fv$
by the standard mollifier $\rho_\eps$,
that is, $\v_k^\eps = \rho_\eps \ast \v_k$, $k=1,2,3$.
Assume that  $\kp\in (0,\frac12)$.
Then, for every $\dl>0$, there exists a number $\sig>0$ and a set $\Om_\dl \subset \Om$, depending also on $\kappa$,  with $\PP(\Om_\dl)>1-\dl$,
 such that for every $\gm=(\gm_1,\gm_2,\gm_3)$ 
with  $|\gm|<\sig$, and for every $\eps>0$,
 the system
\eq{
\lb{sCR}
\pl_{\bar z} r^\eps_1 = 
r_2^\eps \overline{r^\eps_3} - \ovl{r_2^\eps} r^\eps_3 + \gm_1 i \w^\eps,\\
\pl_{\bar z} r^\eps_2= 
\ovl{r_1^\eps} r_3^\eps - r_1^\eps \overline{r_3^\eps} + \gm_2 i \ww^\eps,\\
\pl_{\bar z} r_3^\eps = r_1^\eps \ovl{r_2^\eps}
- \ovl{r_1^\eps}r_2^\eps + \gm_3 i \www^\eps,
\tag{mCR}
}
considered for $\om\in \Om_\dl$,
has a 
solution $r^\eps = (r_1^\eps,r_2^\eps, r_3^\eps) \in \C^\infty(\Tor,\mathbb{C}^3)$  
such that there exists a limit 
\[
r: = \lim_{\eps\to 0} r^\eps \mbox{ in }\C^{-\kp}(\Tor)
\quad \text{for every $\om\in\Om_\dl$.}
\]
Moreover, the convergence holds with the following rate:
\aaa{
\lb{r-rate}
\|r^\eps-r\|_{\C^{-\kp}(\Tor)} \ls \eps^{\frac\kp8} \quad
\text{uniformly in} \;\; \om\in \Om_\dl.
}
\end{theorem}

\begin{definition}
By a solution $r^\eps = (r^\eps _1,r^\eps _2,r^\eps _3)$ 
to system \rf{sCR} we understand a classical solution.
By a solution to \rf{CR} we understand a limit of $r^\eps$, as $\eps\to 0$,
in some sense.
\end{definition}
In other words, Theorem \ref{thm-main-1111} can be rephrased in the following way.
\begin{theorem}\label{thm-main-11110}
Let $(\Om,\mc F, \PP)$ be a probability space and
let $\fv = (\w,\ww,\www)$ be a 3D white noise on $\Tor$ such that $\www$
has zero mean over $\Tor$. Then, equation \rf{CR} has a solution.     
\end{theorem}
\begin{remark}\label{rem-thm-main-11110} Naturally, 
instead of assuming that that the third component $\www$
has zero mean over $\Tor$, we could alternatively assume that the first
$\w$ or  the second component  $\ww$ has zero mean over $\Tor$.
    
\end{remark}

\section{Definitions and notations}

In what follows, we denote by $\Tr^d: = [-\pi,\pi]^d / \!\sim$ 
(where $\sim$ is an equivalence relation that identifies opposite faces of the hypercube $[-\pi,\pi]^d$),
the $d$-dimensional torus which we will treat as a compact Riemannian manifold and a topological group. By $0$ we will denote the corresponding zero element of  $\Tr^d$. 

By $\mc D(\Tr^d)$ we denote the space of $C^\infty$-class functions on $\Tr^d$ and by $\mc D'(\Tr^d)$ the space of distributions on $\Tr^d$, i.e. the topological dual of $\mc D(\Tr^d)$. Note that %contrary to the case of the whole space $\mathbb{R}$, 
the space $\mc D(\Tr^d)$, also called the space of test functions, is a 
Fr\'echet space endowed with the naturally defined countable family of  seminorms. 
If $U$ is an open subset of $\Tr^d$, then by $\mc D(U)$ we denote the class of those test functions  whose support is contained in $U$. Similarly, by $\mc D'(U)$ we 
denote the topological dual of $\mc D(U)$.

%\coma{The authors of \cite[Definition A.9]{bcz2023-1} used test functions $\ffi\in \mc  D(\Tr^d)$ whose support was contained in the ball $B(0,1)$. This was to ensure some uniformity as well to have functions of compact support. I am not sure we need this here since in our case every function has a compact support.} \blue{Evelina: yes, it is better to use the original difinition so one can keep
%the same H\"older-Zygmund norms.}

For a test function $\ffi\in \mc D(\Tr^d)$, $\la>0$ and $x\in \Tr^d$,
we define the rescaled centered function $\ffi^\la_x$ as follows:
\aa{
\ffi^\la_x(y) : = \frac1{\la^d}\,\ffi\Big(\frac{y-x}{\la}\Big).
}
Furthermore, we write $\ffi_x$ instead of $\ffi^1_x$.

%Let $B_1(0)\sub\Rnu^2$ denote the  open unit ball centered in the origin.
Let $r\in \Nnu_0$ and $\alpha  \in \mathbb{R}$. Define the classes
of test functions
\eqn{
\label{eqn-B^r}
&\ms B^r: = \{\ffi\in\mc D(\Tr^d): \|\pl^k\ffi\|_\infty\lt 1, \; \text{for all}\; 0\lt |k| \lt r\},\\
%\label{eqn-B_a}
&\ms B_\al: = \{\ffi\in\mc D(\Tr^d):
\int_{\Tr^d} \ffi(z) z^k dz = 0, \; \text{for all}\; 0\lt |k| \lt \al\},\,
\mbox{ if } \al \gt 0,\\
%\label{eqn-B^r_a}
&\ms B^r_\al: = \ms B^r \cap \ms B_\al,\,
\mbox{ if } \al \gt 0,
}
where $k=(k_1, \ldots, k_d)$ is a multiindex and if $z=(z_1,\dots, z_d)\in\Tr^d$,
$z^k: = z_1^{k_1}\dots z_d^{k_d}$.
\begin{remark}
\rm
Note that our definitions of $\ms B^r$, $\ms B_\al$, and $\ms B^r_\al$ differ from those in \cite[Definition A.9]{bcz2023-1}, as we have taken into account the periodicity of our problem.
\end{remark}
Define the dyadic interval 
\aa{
[0,1]_\D: = \{2^{-n}, \, n\in \Nnu_0\}.
}
\begin{definition}[Periodic H\"older-Zygmund spaces $\ms Z^\al(\Tr^d)$]
\lb{HZ1111}
Assume that $\alpha \in \mathbb{R}$. If $\alpha  <0$,  we denote by   $r=[-\al+1]$ the smallest positive integer $r >-\al$.
We define a ``pseudo-norm" $|f|_{\ms Z^\al}$ for $f \in \mc D'(\Tr^d)$  as follows. \\
If $ \al\gt 0$, then we put 
\aaa{
&|f|_{\ms Z^\al}:=
\sup_{\psi \in \ms B^0} |f(\psi)|
+ \sup_{\substack{x\in\Tr^d\!, \,\la\in [0,1]_\D,\\ \ffi\in \ms B^0_\al}}\frac{|f(\ffi^\la_x)|}{\la^\al}. \lb{HZ1}
}
If $ \al< 0$, then, with $r=[-\al+1]$,  we put 
\aaa{
&|f|_{\ms Z^\al}:=
 \sup_{\substack{x\in\Tr^d,\, \la\in [0,1]_\D,\\ \ffi\in \ms B^r}}\frac{|f(\ffi^\la_x)|}{\la^\al}. \lb{HZ2}
}
We define the periodic H\"older-Zygmund space
$\ms Z^\al: = \ms Z^\al(\Tr^d)$ as the space of distributions $f\in \mc D'(\Tr^d)$
for which  $|f|_{\ms Z^\al}<\infty$. 
\end{definition}
\begin{remark}
\rm
By Lemma  \ref{h-e-1111}, Definition \ref{HZ1111} agrees
with the classical definition of a H\"older-Zygmund
space in the periodic setting (see Remark \ref{intersection-1111}) 
only if $\al\notin \Nnu$.
For $\al\in \Nnu$, $\ms Z^\al(\Tr^d)$ defines a distinct space, 
which is not utilized in our analysis.
\end{remark}
\begin{remark}
\rm
It is known that classical H\"older-Zygmund spaces are 
not separable and that the space $C^\infty$ is not dense 
in them, see \cite{Runst+Sickel_1996}. For this reason, we define so called 
little H\"older-Zygmund spaces $\Zs^\al$.
%work with  defined as the closure in $\Zs^\al$ of $\C^\infty(\Tor)$.
\end{remark}
\begin{definition}[Periodic little H\"older-Zygmund spaces $\Zs^\al(\Tr^d)$]
We define a periodic little H\"older-Zygmund space $\Zs^\al:=\Zs^\al(\Tr^d)$
as the closure of $\C^\infty(\Tr^d)$ in $\ms Z^\al(\Tr^d)$. We furthermore
denote by $|\cdot|_\al$, $\al\in\Rnu$, the restriction
of the norm \rf{HZ1}, for $\al\gt 0$, and \rf{HZ2},  
for $\al < 0$, to $\Zs^\al(\Tr^d)$. 
\end{definition}
\begin{remark}
\rm
It is important to notice that by Lemma \ref{wn1111},
a Gaussian white noise $\v$ on $\Tr^d$, 
as well as its convolutions $K\ast \v$ and $G\ast \v$
(see Definition \ref{green}),
are elements of little H\"older-Zygmund spaces. Furthermore,
by Lemma \ref{young},
Young products are well defined (specifically, uniquely defined)
when they are restricted to little H\"older-Zygmund spaces. 
We also remark that Young products are not uniquely defined
in classical H\"older-Zygmund spaces, 
%introduced in Definition \ref{HZ1111},
see \cite[Remark 14.2]{cz2020}.
For these reasons, little H\"older-Zygmund spaces is a natural choice
for our analysis. 
\end{remark}
\begin{remark}
\rm
For simplicity of notation, 
%until the end of Appendix \ref{secHZ-1111},
we often write $\Zs^\al$, $\ms Z^\al$, $\C^\al$, and $\C^\infty$ 
instead of $\Zs^\al(\Tr^d)$, $\ms Z^\al(\Tr^d)$, 
$\C^\al(\Tr^d)$, and $\C^\infty(\Tr^d)$, respectively.
\end{remark}
% \begin{remark}
% \rm
% It follows from our definition that if $\al$  and $\beta$
% have the same sign or $\beta = 0$, then
% \aaa{
% \lb{zba}
% \ms Z^\al \sub \ms Z^\beta \quad \text{and} \quad 
% \Zs^\al \sub \Zs^\beta \quad \mbox{if} \quad \beta \lt \alpha.
% }
% If $\al$ and $\beta$ have different signs, then \rf{zba}
% is still valid by Lemma \ref{kp1111}.
% \end{remark}

If $f$ is $\Cnu$-valued, then we set
$|f|_\al^2: = |\Re f|^2_\al + |\Im f|^2_\al$. 
%For a $\Cnu^d$-valued function $f$ we define the 
%H\"older-Zygmund norm likewise and use the symbol $\|f\|_\al$ instead of $|f|_\al$.
The above definitions are for $\Cnu$-valued distributions. If $f$ is a $Y$-valued distribution,
%The same holds for $Y$-valued distributions, 
where $Y$ is an $n$-dimensional Euclidean space over $\Cnu$, i.e., 
if $f=(f_1, \ldots, f_n)$, where $f_k$ are $\Cnu$-valued, then
we define the norm
\aaa{
\lb{YHZ-1111}
\|f\|^2_\al: = \sum_{k=1}^n  |f_k|^2_\al.
}
Remark that for $Y$-valued distributions, we use the notation 
$\|\cdot\|_\al$, rather than $|\cdot|_\al$, for respective H\"older-Zygmund norms.
%e.g. $Y=\Cnu^3$.
\begin{remark}
\rm
In \cite[Definition A.9]{bcz2023-1}, the H\"older-Zygmund norms
are defined with respect to sets of test functions supported
on unit balls centered in the origin. 
In Lemma \ref{h-e-1111}, we show that
both definitions are equivalent for $\al\notin \Nnu$, that is, we can equivalently define
the H\"older-Zygmund norms as in Definition \ref{C-3}.
\end{remark}

\begin{definition}[Green function for $\pl_{\bar z}$]
\lb{green}
Let $K$ be the Green function for the Laplacian on $\Tor$. 
We define
\aaa{
\lb{green-1111}
G: =   2\pl_z  K = \pl_1  K - i \pl_2 K. %= :G_1 - i G_2
}
We call $G$ the Green function for the operator $-2\pl_{\bar z}$. Therefore,
$-2G$  is the Green function for the Cauchy-Riemann operator $\pl_{\bar z}$.
%{\it Green function for the Cauchy-Riemann operator} $\pl_{\bar z}$. 
We also put 
 \[G_1: = \pl_1 K \mbox{ and }G_2: = \pl_2 K.
 \]
\end{definition}
%\coma{List properties of $G$ that we use in this paper.
%\aaa{\lb{eqn-G properties}
%\pl_{\bar z} \left( -4 i G \ast u \right)&= 2i  u - [\mbox{average of }u] 
%}
%\blue{Evelina: the properties are listed in the appendix}
%}
%

\section{Solving problem \rf{sCR}}

\subsection{Mild formulation of problem \rf{sCR}}
Let the function $G$ be given by \rf{green-1111}. 
We  define
\aaa{
\lb{eqn-xi^eps}
&\xi^\eps_k: = 2i G \ast \v^\eps_k = 2 G_2 \ast  \v^\eps_k
 + 2i  G_1 \ast  \v^\eps_k, \qquad k=1,2,3.
}
Note that we can write 
\aa{
\xi^\eps_k: = \xi^\eps_{2k} + i \xi^\eps_{1k}.
}
%that is,
%\aaa{\label{eqn-xi^eps}
%\xi^\eps_{11} =  2G_1\ast\w^\eps, \;\; \xi^\eps_{21} = 2G_2\ast\w^\eps, \;\; \xi^\eps_{12} =  2G_1\ast\ww^\eps, \;\; \xi^\eps_{22} =  2G_2 \ast\ww^\eps.
%}
\begin{proposition}
\lb{pro-123-1111}
Let $c>0$,  $\eps>0$ and $\gamma_k \in \mathbb{R}$, $k=1,2,3$.
There exists $\Cnu$-valued Gaussian mean zero random variables $a$ and $b$ such that if
$(r^\eps_1,r^\eps_2,r^\eps_3)$ is an  $\C^\al(\Tor)$-valued random variable, for some  $\al\in (0,1)$,
which is a pathwise solution of the system
\eq{
\lb{sys_r1111}
r^\eps_1= 2G\ast \big(\ovl{r^\eps_2}r^\eps_3-r^\eps_2\ovl{r^\eps_3}\big)- \gm_1\xi^\eps_1 + a,\\
r^\eps_2=  2G\ast \big(r^\eps_1 \ovl{r^\eps_3} - \ovl{r^\eps_1}r^\eps_3\big) - \gm_2\xi^\eps_2 + b,  \\
r^\eps_3 =   2G \ast (\ovl{r^\eps_1}r^\eps_2 - r^\eps_1\ovl{r^\eps_2}) -
\gm_3\xi^\eps_3 + c, 
}
then $r^\eps_k\in\C^\infty(\Tor)$, $k=1,2,3$, and $(r^\eps_1,r^\eps_2,r^\eps_3)$ solves \rf{sCR}.
\end{proposition}

Before we embark on the proof of this proposition we need to state and prove the following auxiliary result. 
\begin{lemma}
\lb{in0-1111}
Assume $\al\in (0,1)$, $g_k\in \C^\al(\Tor,\mathbb{R})$ and  $f_k = iG \ast g_k$, for $k=1,2$.
%(with the H\"older exponent belonging to $(0,1)$). 
Then, in $\C^{\al+1}(\Tor,\mathbb{R})$, 
%\lb{in0-3333}
%Assume $f_k = iG \ast g_k$, for $k=1,2$, where $g_k$ are real-valued
%and $\al$-H\"older continuous functions, $\al\in (0,1)$. Then, 
\aaa{
\lb{in0-3333}
%2i\, \Im f_1 \ovl{f_2} = 
f_1 \ovl{f_2} - \ovl{f_1} f_2 = 
\pl_1 \big( K \ast g_1 \, \pl_2 K \ast g_2\big) - \pl_2 \big(K \ast g_1 \, \pl_1 K \ast g_2\big).
}
and 
\aaa{
\lb{in0-2222}
\int_\Tor \Im f_1 \ovl{f_2}  =  0.
} 
\end{lemma}
\begin{proof}
First of all, we note that for any complex numbers $z_1$ and $z_2$,
\aa{
2i \, \Im z_1 \ovl{z_2} = z_1 \ovl{z_2} - \ovl{z_1} z_2.
}
%We show \rf{in0-2222} for $f_k = iG \ast g_k$, $k=1,2$. 
%The other case can be treated likewise.
Note that 
\aa{
iG = 2 i \pl_z K = i(\pl_1 K - i\pl_2 K) = \pl_2 K + i\pl_1 K,
}
where $K$ is the Green function of the 
Laplacian on $\Tor$, which is real-valued.  Since $g_k$ are $\C^\alpha$,
%Let $\alpha \in (0,1)$ be such that $g_k$ are $C^\alpha$. 
then $K\ast g_k$ are $\C^{2+\alpha}$ and 
\aa{
f_k=iG\ast g_k = \pl_2 K \ast g_k + i\pl_1 K \ast g_k.
}
Thus, we  have 
\mm{
f_1 \ovl{f_2} - \ovl{f_1} f_2 = 2i\, \Im f_1 \ovl{f_2}
= (\pl_1 K \ast g_1) (\pl_2 K \ast g_2) - (\pl_2 K \ast g_1) (\pl_1 K \ast g_2) \\
= \pl_1 \big( K \ast g_1 \, \pl_2 K \ast g_2\big) - (K \ast g_1) \pl^2_{12} K \ast g_2
- \pl_2 \big(K \ast g_1 \, \pl_1 K \ast g_2\big) + (K \ast g_1)  \pl^2_{12} K \ast g_2\\
= \pl_1 \big( K \ast g_1 \, \pl_2 K \ast g_2\big) -
\pl_2 \big(K \ast g_1 \, \pl_1 K \ast g_2\big)
}
which proves \eqref{in0-3333} and \rf{in0-2222}.
\end{proof}
\begin{proof}[Proof of Proposition \ref{pro-123-1111}] 
We start by noticing that since the solution $r^\eps:=(r^\eps_1,r^\eps_2,r^\eps_3)$
to \rf{sCR} is $\C^\al(\Tor)$-valued,
it is automatically $\C^\infty(\Tor)$-valued. Indeed, $(\xi^\eps_1,\xi^\eps_2,\xi^\eps_3)$
is in $\C^\infty(\Tor)$. On the other hand, by Lemmata \ref{fh2222} and \ref{Gconv-1111},
a convolution with $G$ increases the regularity of the expression in parenthesis,
on the right-hand side of \rf{sCR}, by 1, 
inferring that $r^\eps$ is in $\C^{1+\al}(\Tor)$. By induction, we conclude that 
$r^\eps$  is in $\C^\infty(\Tor)$.

%Define $c= \sg e^{i\te}$, $\te \ne n \pi$, $n\in \Nnu$, 
In what follows, with a slight abuse of notation, 
we omit the index $\eps$, that is, we write $\w$, $\ww$ and $\www$
instead of $\w^\eps$, $\ww^\eps$ and $\www^\eps$. The same applies to
other objects in \rf{sCR}.

Define $a$ and $b$ by the formulas
\aaa{ 
\lb{ab-1111}
2 a c =   i \gm_2 \eta_2, \qquad 
2 b c =  - i \gm_1\eta_1,
}
where 
 $\eta_1$ and $\eta_2$
%in Lemma \ref{noise-f-2222}  
are independent standard real Gaussian random variables
as in Definition \ref{white-noise}.
From now on, we choose and fix $\omega \in \Omega$ so that $\eta_k=\eta_k(\omega )$ and $\xi_k=\xi_k(\omega)$ are real numbers. 
%\adda{Note that these are gaussian random variables.}
Let us introduce $\rho_k$, $k=1,2,3$, as follows:
 %the following random  variables,
\aaa{
\lb{r123-1111}
r_1 = \ro_1  + a, \quad  r_2 =  \ro_2 + b, \quad  
r_3 = \ro_3 + c.
}
%We will search for a solution to the system
Since $(r_1,r_2,r_3)\in\C^\al(\Tor)$,  we infer that 
$(\ro_1,\ro_2,\ro_3)\in\C^\al(\Tor)$ and 
 $(\ro_1, \ro_2, \ro_3)$ is a   solution to the following system 
\dela{\eqref{sys_r11117}, where $a$ and $b$ are given by \rf{ab-1111}:}
\eq{
\lb{sys_r11117}
\ro_1= - 4i\, G\ast \Im 
\big[(\ro_2 + b)(\ovl{\ro_3} +c)\big] - \gm_1 \xi_1, \\
\ro_2= 4i\, G\ast \Im 
\big[(\ro_1 + a)(\ovl{\ro_3} +c)\big] - \gm_2  \xi_2,   \\
\ro_3 = - 4i \, G \ast \Im\big[ (\ro_1 + a)(\ovl{\ro_2} - b)\big] - \gm_3 \xi_3.
%\big[(\ovl{\ro_1} - a)(\ro_2 + b) - (\ro_1 + a)(\ovl{\ro_2} - b)\big] 
}
Then, it follows from \eqref{eqn-xi^eps}  and  \eqref{sys_r11117} that 
we can find    functions $g_j\in \C^\al(\Tor,\mathbb{R})$ satisfying the assumptions of Lemma \ref{in0-1111} and 
\aa{
\ro_j = i G \ast g_j, \quad 
j=1,2,3.
}

Hence, by Lemma \ref{in0-1111},
for all $k,j = 1,2,3$, $k\ne j$,
\aa{
\int \Im \ro_k \ovl{\ro_j} = 0.
}
In addition, we notice that $\int_\Tor  \ro_j  = 0$
for $j=1,2,3$.
By Lemma \ref{Lav1111},
%(\rho_1,\rho_2,\rho_3)\in\C^{1+\al}(\Tor)$ and
\eqn{
\lb{ii-1111}
\pl_{\bar z} \ro_1 & = 2i\, 
\Im [(\ro_2+b)(\ovl{\ro_3} + c)] +\gm_1 i\,(\w - \eta_1)  -  2  b c\\
& = 2i\, \Im [(\ro_2+b)(\ovl{\ro_3} + c)]
+\gm_1 i\, \w,\\
-\pl_{\bar z} \ro_2 & =  
2i\, \Im [(\ro_1+a)(\ovl{\ro_3} + c)] 
-\gm_2 i\, (\ww - \eta_2) - 2 a c \\
& = 2i\, \Im [(\ro_1+a)(\ovl{\ro_3} + c)] 
 -\gm_2 i\, \ww,
\\
\pl_{\bar z} \ro_3 & =  2i \,
\Im (\ro_1 + a) (\ovl{\ro_2} - b)  %- 2i \, \Im a \ovl b
+ \gm_3 i\, \www .
}
By the change of variable \rf{r123-1111}, 
$(r_1,r_2,r_3)$ solves \rf{sCR}.
\end{proof}
\begin{corollary}
\lb{abc0}
If $\w$ and $\ww$ have zero means over $\Tor$,
then the statement of Proposition \ref{pro-123-1111} 
holds with $a=b=c=0$.
\end{corollary}
\begin{proof}
Recall that the fact that $\w$ and $\ww$ have zero means over $\Tor$
means that $\eta_1=\eta_2 = 0$. Therefore, 
system \rf{ii-1111} (and therefore system \rf{sCR}) is satisfied with $a=b=c=0$.
\end{proof}

\subsection{The fixed-point argument}
\begin{lemma}
\lb{lem-R-3333}
Throughout this lemma, we  choose and fix \(\gm = (\gm_1, \gm_2, \gm_3)\).
For \(\eps > 0\), let
\aaa{
\lb{Rxi2222}
R^\eps: = r^\eps + \gm \xi^\eps,
}
where $\gm \xi^\eps: = (\gm_1 \xi^\eps_1,\gm_2 \xi^\eps_2, \gm_3 \xi^\eps_3)$.
Then, system  \rf{sys_r1111} is equivalent to
\aaa{
\lb{R-eps-1111}
R^\eps = -2 G \ast \big(R^\eps \x \ovl {R^\eps} 
 - (\gm  \xi^\eps )\x \ovl R^\eps   -  R^\eps \x (\gm \ovl {\xi^\eps} ) \big)
 - \td \gm \zeta^\eps +\vect{a}{b}{c},
}
where 
\aaa{\lb{eqn-zeta}
\zeta^\eps : = 2G \ast (\xi^\eps \x \ovl {\xi^\eps} ),
 \qquad 
 %\text{and} \quad
\td \gm :  = (\gm_2 \gm_3, \gm_1\gm_3, \gm_1\gm_2)
}
and by $\td \gm\zeta^\eps$, we mean 
$(\td \gm_1\zeta^\eps_1, \td \gm_2 \zeta^\eps_2, \td \gm_3\zeta^\eps_3)$.
\end{lemma}
\begin{remark}
Note that so far $\zeta^\eps$ makes sense only for $\eps>0$ because the product $\xi^\eps \x \ovl {\xi^\eps}$ is ill-defined for $\eps=0$.
\end{remark}
\begin{proof}[Proof of Lemma \ref{lem-R-3333}]
One just needs to notice that \rf{sys_r1111}
is equivalent to 
\aa{
r^\eps = -2G \ast \big(r^\eps \x \ovl{r^\eps}\big) +\vect{a}{b}{c}.
}
The substituting $r^\eps = R^\eps - \gm \xi^\eps$, we obtain 
\rf{R-eps-1111}. The proof is complete.
\end{proof}
For every $\nu\in \Rnu$, introduce the Banach space 
\aa{
E_{\nu} = \big\{R = (R_1,R_2,R_3):\Tor\to \Cnu^3: \;\; R_k \in \Zs^\nu(\Tor), k=1,2,3\big\}
}
with the norm $\|\cdot\|_\nu$ defined by \rf{YHZ-1111}. 
For every $\sg>0$, we set
\aa{
c= \frac\sg4 \quad \text{and define $a,b$ via $c$ by formulas
\rf{ab-1111}.}
}
Given $\eps > 0$, $\sigma>0$, and $\kp\in (0,\frac12)$, define  %fixed-point 
a map $\Gm_{\eps,\sg}: E_{1-\kp}\to E_{1-\kp}$,
\aaa{
\lb{map-Gm}
\Gm_{\eps,\sg} R: =
-2G \ast \Bigl(R\x \ovl R+ (\gm  \xi^\eps)\x \ovl R + R\x (\gm \ovl {\xi^\eps})\Bigr) 
 - \td \gm \zeta^\eps + \vect{a}{b}{c}.
}
As we will see from Lemmata \ref{lem8888}
and \ref{zeta-lemma-1111}, $\xi^\eps$ and $\zeta^\eps$ are well defined
also for $\eps=0$. Therefore, \rf{map-Gm} also defines a map
$\Gm_{0,\sg}$ for every $\sg>0$.
Our aim is to  show that for every $\dl>0$,
there exists a set $\Om_\dl\sub\Om$, $\PP(\Om_\dl)>1-\dl$, 
and
$\sg>0$ such that for every $\eps \gt 0$ and $|\gm| \ls \sg^2$, for every $\om\in\Om_\dl$, 
the map $\Gm_{\eps,\sg}: E_{1-\kp}\to E_{1-\kp}$ is well defined and has a unique fixed point.
%for every $\om\in\Om_\dl$.
 
For this aim, as well as to define the map $\Gm_{0,\sg}$, we need the following 
 auxiliary Lemmata \ref{lem8888} and \ref{zeta-lemma-1111}.
 \begin{remark}
 \lb{remark-3.5}
Note that  by Lemma \ref{smooth}, $\C^\infty(\Tor) = \medcap_\nu \Zs^\nu(\Tor)$.
Therefore,
\aa{
\fv^\eps, \xi^\eps, \zeta^\eps \in E_\nu \quad \text{for all} \;\;
\nu \in \Rnu \;\; \text{and} \;\; \eps>0.
}
\end{remark}
\begin{lemma}%{lem}
\lb{lem8888}
For all $\kp\in (0,1)$,  $\xi \in E_{-\kp}$.
Moreover, for all $\eps\gt 0$,
\aa{
\|\xi^\eps\|_{-\kp} \ls \|\fv\|_{-1-\kp},
}
where $\fv$ on the right-hand side is the non-mollified 3D white noise 
appearing in \rf{CR},
$\xi^\eps$ is defined by \eqref{eqn-xi^eps}, $\xi^0: = \xi$,
and the constant in the above inequality
does not depend on $\eps\gt 0$.
\end{lemma}%{lem}
\begin{proof} %[Proof of Lemma \ref{lem8888}]
%Let $\kp\in (0,1)$.
By Lemma \ref{wn1111}, $\v_k$, $k=1,2,3$, are
in $\Zs^{-1-\kp}(\Tor)$. Furthermore, by Lemma \ref{Gconv-1111},
$\xi_k$, $k=1,2,3$, are in $\Zs^{-\kp}(\Tor)$.
Since $\xi^\eps_{kj} = G_k \ast \v^\eps_j$, $k = 1,2$; $j=1,2,3$, one has
\aa{
|\xi^\eps_{kj}|_{-\kp} \ls |\v^\eps_j|_{-1-\kp} \ls |\v_j|_{-1-\kp}. 
}
In particular, we have used Lemma \ref{lem22-22}.
\end{proof}

\begin{lemma}
\lb{zeta-lemma-1111} 
Let, for $\eps > 0$,  $\zeta^\eps$ be defined by \eqref{eqn-zeta}
and $\kp\in (0,\frac12)$. Then 
%$\zeta^\eps \in E_{1-\kp}$ and 
$\zeta^\eps$ 
has a limit $\zeta$ in $E_{1-\kp}$ as $\eps \to 0$. Furthermore,
% there exists a constant $K>0$, which depends only on $\kp$
% and, in particular, does not depend on $\eps$,
% such that 
for all $\eps\gt 0$ and $\kp\in (0,\frac12)$,
\aaa{
\lb{bound-zeta}
\|\zeta^\eps\|_{1-\kp} \ls \|\fv\|_{-1-\frac{\kp}4}^2, 
}
where $\zeta^0:=\zeta$, and the constant in the above inequality
does not depend on $\eps\gt 0$.
\end{lemma}
Lemma \ref{zeta-lemma-1111} is a fundamental result
of this work, underpinning the proof
of Theorem \ref{thm-main-1111}.
Its proof is postponed to Subsection \ref{zeta},
where we give the explicit expression for $\zeta$ and  prove
\rf{bound-zeta}.
\begin{lemma}\label{lem-3.6}
If  $\eps\gt 0$, $\sg>0$, and $\kp\in (0,\frac12)$,
then the map $\Gm_{\eps,\sg}: E_{1-\kp}\to E_{1-\kp}$, given by \rf{map-Gm},
 is well-defined.
\end{lemma}%{lem}
% \begin{remark}
% \lb{3.9}
% By $\xi^0$ and $\zeta^0$, we understand the limits of
% $\xi^\eps$ and $\zeta^\eps$ in $E_{-\kp}$ and $E_{1-\kp}$,
% respectively.
% \end{remark}
\begin{proof}[Proof of Lemma \ref{lem-3.6}]
Let us choose and fix $R\in E_{1-\kp}$. In view of Lemmata \ref{zeta-lemma-1111} and  \ref{Gconv-1111}, it is sufficient to show that 
$R\x \ovl R+ (\gm  \xi^\eps)\x \ovl R + R\x (\gm \ovl {\xi^\eps}) \in E_{-\kp}$.

Since, $\overline{R}\in E_{1-\kp}$ and   the space $E_{1-\kp}$ is an algebra,  we infer that $R \times \overline{R}\in E_{1-\kp} \subset E_{-\kp}$. 
Moreover, by Lemma \ref{young} and since ${\xi^\eps} \in E_{-\kp}$, we infer that both 
$(\gm  \xi^\eps)\x \ovl R$ and   $R\x (\gm \ovl {\xi^\eps})$ belong to $  E_{-\kp}$. The proof is complete. 
\dela{
Then, since 
\aa{
\|R\x \ovl R\|_{1-\kp} \ls \|R\|_{1-\kp}^2.
}
we  have
\aaa{
\lb{G-est-1111}
\|\Gm_{\eps,\sg} R\|_{1-\kp} 
\ls \|R\|^2_{1-\kp} +  |\gm|\|\xi^\eps\|_{-\kp} \|R\|_{1-\kp} 
 + \td \gm\|\zeta^\eps\|_{1-\kp} + |a|+ |b| + c.
% \ls \|R\|^2_{1-\kp} + \|\xi^\eps\|_{-\kp}^2  + 
% \|\xi^\eps\|_{-\kp} + \|R\|_{1-\kp} + \|\zeta^\eps\|_{1-\kp}.
}
The first term on the right-hand side appears due to Lemma \ref{kp1111}. 
%Indeed, \aaa{\lb{Rkp-1111} \|R\|_{-\kp}\ls \|R\|_{1-\kp},}
 
The constant in the estimate \rf{G-est-1111}
depends on the kernel $G$, on the constant
in the inequality \rf{Y-torus-1111}  for Young products,
and on the constant in \rf{Rkp-1111}. Specifically,
it depends on $\kp$ through these constants.}
\end{proof}
Let us consider the closed ball in $E_{1-\kp}$ of radius $\sg>0$ (to be fixed later)
\aa{
M_{1-\kp, \sg}: = \{R\in E_{1-\kp}: \, \|R\|_{1-\kp} \lt \sg\}.
} Obviously, $M_{1-\kp, \sg}$ is  a complete metric space.

\begin{proposition}%{pro}
\lb{pro2222}
For every $\dl\in (0,1)$ and $\kp\in (0,\frac12)$, there exists a set
 $\Om_\dl\sub\Om$, $\PP(\Om_\dl)>1-\dl$, and
a number $\sg>0$ such that  
for all $\om\in\Om_\dl$,  $\eps \gt 0$, 
and $\gm=(\gm_1,\gm_2,\gm_3)$ %$(\gm_j>0,$ $j=1,2,3)$
 satisfying $|\gm|\ls\sg^2$,  $\Gm_{\eps,\sg}$   maps $M_{1-\kp, \sg}$ into itself and the 
 corresponding map % restriction and \red{co-restriction ??} 
$\Gm_{\eps,\sg}: M_{1-\kp, \sg}\to M_{1-\kp, \sg}$ is  a strict contraction 
and hence has a unique fixed point in $M_{1-\kp, \sg}$.
Moreover, the contraction constant of $\Gm_{\eps,\sg}$ does not depend 
on $\eps \gt 0$ and $\om\in\Om_\dl$.
\end{proposition}%{pro}
\begin{remark}
\rm
We reinforce that
Proposition \ref{pro2222} is also valid for the map 
$\Gm_{0,\sg}$; $\xi^0$ and $\zeta^0$ are understood as in
Lemmata \ref{lem8888} and \ref{zeta-lemma-1111}.
\end{remark}
\begin{proof}[Proof of Proposition \ref{pro2222}]
Let $\dl\in (0,1)$ be fixed arbitrarily. Since
the Gaussian random variables $\eta_1$ and $\eta_2$  are finite a.s., by Lemma \ref{pro1111}  we
 can find a $\La\in (0,+\infty)$ such that
$\PP(\Om_\dl) > 1-\dl$, where
\aaa{
\lb{omdl1111}
\Om_\dl: = \Big\{\om: \|\fv\|_{-1-\kp} + |\eta_1| + |\eta_2| 
< \La\Big\}.
}
Note that  by Lemmata \ref{lem8888}, \ref{zeta-lemma-1111} and \ref{lem22-22},
$\|\xi^\eps\|_{-\kp}$ and $\|\zeta^\eps\|_{1-\kp}$ are bounded
on $\Om_\dl$ uniformly in $\eps$ by a constant depending only on $\La$.
Take $R\in M_{1-\kp, \sg}$. Recall that in the definition of  
the map $\Gm_{\eps,\sg}$ by \rf{map-Gm}, we took $c=\frac\sg4$.
Choose $\gm = (\gm_1,\gm_2,\gm_3)$ such that $|\gm|< \frac{\sg^2}{16\La}$.
Then, $|a|+|b|+c<\frac\sg2$.
By Lemmata \ref{lem8888}, \ref{zeta-lemma-1111}, and \ref{young},
\mm{
\|\Gm_{\eps,\sg} R\|_{1-\kp} 
\lt C\big(\sg^2+|\gm| \sg \|\xi^\eps\|_{-\kp} 
+ |\gm|^2\|\zeta^\eps\|_{1-\kp}\big) +\frac\sg2
\lt C_\La (\sg^2  + \sg^3 + \sg^4) +\frac\sg2 \lt \sg,
}
where $C$ is a constant that depends only on $\kp$ and $G$, 
and $C_\La$ is another constant that 
depends only on $C$ and $\La$. 
The latter inequality is achieved by choosing $\sg$ such that 
\aaa{
\lb{fp-1111}
C_\La  (\sg  + \sg^2 + \sg^3) <\frac12.
}
Therefore, $\Gm_{\eps,\sg}$ maps $M_{1-\kp, \sg}$ to $M_{1-\kp, \sg}$.
Further, take two points $R,\hat R\in M_{1-\kp,\sg}$.
We obtain
\aaa{
\lb{ineq1111}
\|\Gm_{\eps,\sg} R - \Gm_{\eps,\sg} \hat R\|_{1-\kp} \lt
2C( \sg + |\gm|\|\xi^\eps\|_{-\kp}) \|R-\hat R\|_{1-\kp}
\lt 2C_\La  (\sg + \sg^2)  \|R-\hat R\|_{1-\kp}.
 }
 Observe that by \rf{fp-1111}, $2C_\La  (\sg + \sg^2)<1$. 
 %where the constant $C$ and $C_\La$ are the same as in \rf{G-est-1111}.
 Therefore, $\Gm_{\eps,\sg}$ has a unique fixed point in $M_{1-\kp, \sg}$.
 We also remark that the fixed-point constant does not depend on 
 $\eps\gt 0$ and $\om\in \Om_\dl$.
\end{proof}

 \begin{remark}
 \rm
 We remark that the proof holds also for $\eps=0$. In this case,
 $\xi^0$ and $\zeta^0$ are the limits of $\xi^\eps$ and $\zeta^\eps$  in $E_{-\kp}$ 
 and $E_{1-\kp}$, respecively.
 \end{remark}

\section{Renormalization and regularity of $\zeta^\eps$.}
\lb{zeta}
Recall that $\zeta^\eps$, defined for $\eps>0$ as
\aa{
\zeta^\eps: = 2G \ast (\xi^\eps \x \ovl {\xi^\eps} ),
}
is a noise term in the fixed-point equation \rf{lem-R-3333}. We also recall
that the proof of Lemma \ref{zeta-lemma-1111} was postponed to this section.
According to Lemma \ref{zeta-lemma-1111}, $\zeta^\eps$ has a limit $\zeta$ in $E_{1-\kp}$. This is a non-trivial and key result of this work
since, as it will be clear from Lemma \ref{lc1111},
%\ref{te-rn-1111},
the explicit expression for $\xi^\eps \x \ovl {\xi^\eps}$ suggests 
that $\zeta\in E_{-\kp}$.
The latter property, however, is not sufficient for the fixed-point argument to work. 
The first step in proving Lemma \ref{zeta-lemma-1111} is Lemma \ref{lc1111} in which, in particular, we prove that the $ \tet:= \lim_{\eps\to 0} \xi^\eps \x \ovl {\xi^\eps}$ exists 
in  $\mc D'(\Tor)$. The second step is Proposition \ref{prop-2222},
where we prove that
$\tet\in \cap_{\kp\in (0,\frac12)}E_{-\kp}$. Corollary \ref{cor-2222}, which is a direct of consequence of 
Proposition \ref{prop-2222}, implies Lemma \ref{zeta-lemma-1111}.

\begin{lemma}
\lb{in0-4444}
Let $A\in\Rnu$ be a constant. Then, 
under the assumptions of Lemma \ref{in0-1111}, 
in $\C^{\al+1}(\Tor,\mathbb{R})$, we have the identity 
\aa{
f_1 \ovl{f_2} - \ovl{f_1} f_2 = 
\pl_1 \big( (K \ast g_1 - A) \, \pl_2 K \ast g_2\big) -
\pl_2 \big( (K \ast g_1 - A) \, \pl_1 K \ast g_2\big).
}
\end{lemma}
\begin{proof}
Using the definition of $G$ and $f_k$, $k=1,2$, we obtain
\mm{
f_1 \ovl{f_2} - \ovl{f_1} f_2 
= \big(\pl_1 (K \ast g_1 - A)\big) (\pl_2 K \ast g_2) - 
\big(\pl_2 (K \ast g_1-A)\big) (\pl_1 K \ast g_2) 
\\
= \pl_1 \big( (K \ast g_1 - A) \, \pl_2 K \ast g_2\big) - (K \ast g_1) \pl^2_{12} K \ast g_2
- \pl_2 \big((K \ast g_1 -A) \, \pl_1 K \ast g_2\big) + (K \ast g_1)  \pl^2_{12} K \ast g_2\\
= \pl_1 \big( (K \ast g_1 - A) \, \pl_2 K \ast g_2\big) -
\pl_2 \big( (K \ast g_1 - A) \, \pl_1 K \ast g_2\big).
}
\end{proof}
% The following result provides an explicit expression for $\xi^\eps \x \ovl{\xi^\eps}$
% since a priori, the components of the aforementioned vector product are expressed as
% ill-defined product of noises.

\begin{lemma}
\lb{lc1111}
The components of $\tet^\eps: = \xi^\eps \x \ovl{\xi^\eps}$  
satisfy the following relations:
\eqn{
\lb{lcxxxx-0} 
\tet_1^\eps = 2 \big(\pl_1[(K\ast \ww^\eps) \xi^\eps_{23}] -
 \pl_2 [(K\ast \ww^\eps)\xi^\eps_{13}]\big),\\
\tet_2^\eps = 2 \big(\pl_1[(K\ast \www^\eps)\xi^\eps_{21}]
- \pl_2 [(K\ast \www^\eps)\xi^\eps_{11}]\big),\\
\tet_3^\eps = 2 \big(\pl_1[(K\ast \w^\eps)\xi^\eps_{22}]
 - \pl_2 [(K\ast \w^\eps)\xi^\eps_{12}]\big),
}
where $\xi^\eps_{jk}$ have been introduced in \eqref{eqn-xi^eps}.
Furthermore, $\tet^\eps$ satisfy, for any $x\in \Tor$, 
the additional relations 
\eqn{
\lb{lcxxxx} 
&\tet_1^\eps = 2 \big(\pl_1[(K\ast \ww^\eps  - K \ast \ww^\eps(x))\xi^\eps_{23}] 
- \pl_2 [(K\ast \ww^\eps - K \ast \ww^\eps(x))\xi^\eps_{13}]\big),\\
&\tet_2^\eps = 2 \big(\pl_1[(K\ast \www^\eps - K \ast \www^\eps(x))\xi^\eps_{21}] 
- \pl_2 [(K\ast \www^\eps - K \ast \www^\eps(x)) \xi^\eps_{11}]\big),\\
&\tet_3^\eps = 2 \big(\pl_1[(K\ast \w^\eps  - K \ast \w^\eps(x))\xi^\eps_{22}]
- \pl_2 [(K\ast \w^\eps - K \ast \w^\eps(x))\xi^\eps_{12}]\big),
}
where both \rf{lcxxxx-0} and \rf{lcxxxx}  are valid in $\C^\infty(\Tor)$.
Moreover, there exists a limit 
\begin{equation}
    \tet: = \lim_{\eps\to 0} \tet^\eps
\mbox{ in } \mc D'(\Tor),
\end{equation}
and the  components $\tet_k$ of $\tet$, $k=1,2,3$, are given by
 \eqn{
 \lb{tet-d-1111}
\tet_1 = 2 \big(\pl_1[(K\ast \ww) \xi_{23}] -
 \pl_2 [(K\ast \ww) \xi_{13}]\big),\\
\tet_2 = 2 \big(\pl_1[(K\ast \www)\xi_{21}] 
- \pl_2 [(K\ast \www) \xi_{11}]\big),\\
\tet_3 = 2 \big(\pl_1[(K\ast \w) \xi_{22}]
 - \pl_2 [(K\ast \w)\xi_{12}]\big),
}
and also by
\eqn{
\lb{lcxxxx-lim}
&\tet_1 = 2 \big(\pl_1[(K\ast \ww  - K \ast \ww(x))\xi_{23}] 
- \pl_2 [(K\ast \ww - K \ast \ww(x))\xi_{13}]\big),\\
&\tet_2 = 2 \big(\pl_1[(K\ast \www - K \ast \www(x))\xi_{21}] 
- \pl_2 [(K\ast \www - K \ast \www(x)) \xi_{11}]\big),\\
&\tet_3 = 2 \big(\pl_1[(K\ast \w  - K \ast \w(x))\xi_{22}]
- \pl_2 [(K\ast \w - K \ast \w(x))\xi_{12}]\big),
}
where identities \rf{lcxxxx-lim} are valid for any $x\in \Tor$, 
and the products in the square brackets, in both \rf{tet-d-1111} and \rf {lcxxxx-lim},
are understood as Young products (in the sense of Lemma \ref{young}).
\end{lemma}
\begin{proof}
We have
\aa{
\tet_1^\eps = \xi^\eps_2 \ovl{\xi^\eps_3} - \xi^\eps_3 \ovl{\xi^\eps_2}, 
\qquad \xi^\eps_k = i G \ast (2\v_k^{\eps}), \;\; k=2,3.
}
% By Lemma \ref{in0-1111},
% \aa{
% \xi^\eps_2 \ovl{\xi^\eps_3} - \xi^\eps_3 \ovl{\xi^\eps_2}
% = 4\pl_1(K\ast \ww^\eps \,\pl_2 K \ast \www^\eps)
% - 4\pl_2 (K\ast \ww^\eps \,\pl_1 K \ast \www^\eps).
% }
% The proof of  the existence of the limit $\tet$
% and of \rf{tet-d-1111} is postponed to Lemma \ref{lc1111}.
By Lemma \ref{in0-4444}, for every $A\in\Rnu$,
\aa{
\xi^\eps_2 \ovl{\xi^\eps_3} - \xi^\eps_3 \ovl{\xi^\eps_2}
= 4\pl_1[(K\ast \ww^\eps - A)\pl_2 K \ast \www^\eps]
- 4\pl_2 [(K\ast \ww^\eps -A) \pl_1 K \ast \www^\eps].
}
Recalling that $2\pl_2 K \ast \www^\eps = \xi^\eps_{23}$,
$2\pl_1 K \ast \www^\eps  = \xi^\eps_{13}$, and setting $A=0$,
we obtain the expression for $\tet_1^\eps$ in \rf{lcxxxx-0}.
The identities for $\tet_2^\eps$ and $\tet_3^\eps$ in \rf{lcxxxx-0}
can be obtained likewise.

Further, by choosing and fixing $x \in \Tor$,  and then
setting  $A=K\ast \ww^\eps(x)$, we deduce 
the first identity in \rf{lcxxxx}. 
The last two identities in \rf{lcxxxx}
can be obtained likewise.

 Finally, for any test function $\ffi$,
 \mm{
 \lim_{\eps\to 0} 
 \tet^\eps_1(\ffi)
 = \lim_{\eps \to 0} \int \big(K \ast \ww^\eps - A\big) \xi^\eps_{23}\, \pl_1 \ffi
 + \lim_{\eps \to 0} \int \big(K \ast \ww^\eps - A\big) \xi^\eps_{13} \,\pl_2 \ffi\\
 =  \int \big(K \ast \ww - A\big) \xi_{23}\, \pl_1 \ffi
 + \int \big(K \ast \ww - A\big) \xi_{13}\, \pl_2 \ffi.
 }
 Setting $A=0$, we obtain the first expression in \rf{tet-d-1111}.
 Next, setting  $A=K\ast \ww^\eps(x)$, where 
 $x\in\Tor$ is arbitrary, we obtain the first expression 
 in \rf{lcxxxx-lim}. We also remark that by Lemma \ref{pro1111},
 $(K \ast \ww - A)\in \Zs^{1-\kp}$ 
 and $\xi_{13},\xi_{23}\in \Zs^{-\kp}$, $\kp\in (0,\frac12)$.
 The last two expressions in 
 \rf{tet-d-1111} and \rf{lcxxxx-lim} can be obtained likewise.
\end{proof}
\begin{remark}
\rm
Identities \rf{lcxxxx} and \rf{lcxxxx-lim} are important
to compute the regularity of $\tet$. Namely, using 
these identities, we will show that $\tet\in E_{-\kp}$ and 
$\tet^\eps\to \tet$ in $E_{-\kp}$.
\end{remark}
% because of $2^{1-2\kp}$, $\kp\in (0,\frac12)$
%$\frac\kp2 \in (0,\frac14)$.
% $1- \frac\kp2 \in (\kp,1)$, $1- \frac\kp2 > \kp$, $\frac32 \kp < 1$,
%$\kp< \frac23$, $\kp\in (0,\frac23)$ which is bigger than the previous interval
\begin{proposition}
\lb{lem4444-evelina}
Let  $\td{\ms B}^1$ be the class
of  test functions introduced in Definition \ref{def-h-1111}.
Assume that $\kp\in (0,\frac12)$, $\nu\in (2\kp,1)$, 
$\eta\in \Zs^{-\kp}(\Tr^d)$ and $u \in \Zs^\nu(\Tr^d)$. Then, 
\begin{align}
\lb{eqn-new-01}
\sup_{\substack{x\in \Tr^d, \la\in (0,1], \\ \ffi,\psi\in \td{\ms B}^1}}
\frac{\big|\big[(u  - u(x))\eta\big] \big((\ffi\psi)^\la_x\big)\big|}{\la^{\nu -3\kp}}
\lt [u]_\nu |\eta|_{-\kp} \ls |u|_\nu |\eta|_{-\kp},   
\end{align}
where 
\[
[u]_\nu:= \sup_{\{x\in \Tr^d, |y|\lt 1\}} \frac{|u(x+y)-u(x)|}{|y|^\nu}
\]
and  the product $(u-u(x))\eta$ is understood as the Young product
in the sense of Lemma \ref{young}.
\end{proposition}
%\begin{remark}
%\rm
% By classical argument, 
% it follows that \rf{eqn-new-01} is fulfilled
% for every $\eta\in \Zs^{-\kp}(\Tr^d)$ 
% and every $u$ belonging to the closure of 
% $\Zs^\kap(\Tr^d)$ in $\Zs^\nu(\Tr^d)$.
%\end{remark}
\begin{remark}
We borrowed the type of expression $(u  - u(x))\eta$
from \cite{bcz2023-2}. It has a connection with 
the idea of germs, see \cite{bcz2023-1,bcz2023-2}. 
\end{remark}

\begin{proof}[Proof of Proposition \ref{lem4444-evelina}] 
Take two test functions $\ffi,\psi \in \td {\ms B^1}$.
 Let
\aa{
g(x,\la,y) = \frac{u(y)  - u(x)}{\la^{\nu-2\kp}}.
}
Since $u \in \Zs^\nu(\Tr^d)$ and $\nu>2\kp$,   $|u|_{2\kp} \lt |u|_\nu$. Therefore,
\mm{
\big|[g(x,\la, \fdot) \eta]\big((\ffi\psi)^\la_x\big)\big| =
\Big|\Big[\ffi\Big(\frac{\fdot-x}{\la}\Big)g(x,\la, \fdot) \eta\Big](\psi^\la_x)\Big| {\la^\kp} \la^{-\kp}\\
\lt   \sup_{\hat x\in\Tr^d, \, \hat \la\in (0,1], \hat \psi\in \td {\ms B^1}}
\Big|\Big[\ffi\Big(\frac{\fdot-x}{\la}\Big)g(x,\la, \fdot) \eta\Big]({\hat \psi}^{\hat \la}_{\hat x})\Big| 
{\hat \la^\kp} \la^{-\kp}\\
\ls \Big|\ffi\Big(\frac{\fdot-x}{\la}\Big)g(x,\la, \fdot)\Big|_{2\kp}|\eta|_{-\kp} \,\la^{-\kp},
}
where the constant on the right-hand side depends on $\kp$ and $d$.
Since $\ffi\big(\frac{\fdot-x}{\la}\big)$ is supported on $B_\la(x)$,
\aa{
\Big\|\ffi\Big(\frac{\fdot-x}{\la}\Big)g(x,\la, \fdot)\Big\|_\infty \lt \|\ffi\|_\infty [u]_{\nu-2\kp}.
}
Let $y,z\in \ovl{B_\la(x)}$, so that $|y-z|\lt 2\la$. We have
\eqn{
\lb{compute-1111}
&\Big|\ffi\Big(\frac{y-x}{\la}\Big)\frac{u(y)  - u(x)}{\la^{\nu-2\kp}} 
- \ffi\Big(\frac{z-x}{\la}\Big)\frac{u(z)  - u(x)}{\la^{\nu-2\kp}}\Big| \, |y-z|^{-2\kp}\\
& \lt \Big|\ffi\Big(\frac{y-x}{\la}\Big)\Big| \, \frac{|u(y) - u(z)|}{\la^{\nu-2\kp} |y-z|^{2\kp}}
+ \frac{|u(z)  - u(x)|}{\la^{\nu-2\kp}}
\frac{\big|\ffi\big(\frac{y-x}{\la}\big) - \ffi\big(\frac{z-x}{\la}\big)\big|}{|y-z|^{2\kp}}\\
& \lt \|\ffi\|_\infty [u]_{\nu} \, \frac{|y-z|^{\nu}}{\la^{\nu-2\kp} |y-z|^{2\kp}} + 
 \|\nab\ffi\|_\infty \frac{|u(z)  - u(x)|}{\la^{\nu-2\kp}} \frac{|y-z|^{1-2\kp}}{\la}\\
 & \lt 2^{\nu-2\kp} \, [u]_{\nu} + \sqrt{d}\, 2^{1-2\kp} \, \frac{\la^{1-2\kp}}{\la} \frac{|u(z)  - u(x)|}{\la^{\nu-2\kp}}
  \lt [u]_{\nu} (2^{\nu-2\kp} + \sqrt{d}\, 2^{1-2\kp}).
}
Hence,
\aa{
\Big|\ffi\Big(\frac{\fdot-x}{\la}\Big)g(x,\la, \fdot)\Big|_{2\kp} \ls [u]_\nu \ls |u|_\nu.
}
Now let $y \in \ovl{B_\la(x)}$ and $z\notin \ovl{B_\la(x)}$. Let $z': = [y,z] \cap \pl B_\la(x)$.
We have $|y-z| > |y-z'|$, and hence $|y-z|^{-2\kp} < |y-z'|^{-2\kp}$.
On the other hand, 
\aa{
\ffi\Big(\frac{z-x}{\la}\Big) \frac{u(z)  - u(x)}{\la^{\nu-2\kp}} 
= \ffi\Big(\frac{z'-x}{\la}\Big) \frac{u(z')  - u(x)}{\la^{\nu-2\kp}}= 0.
}
Therefore, the first line in \rf{compute-1111} can be evaluated
from above by the identical line with $z'$ in place of $z$. 
Since $z' \in \ovl{B_\la(x)}$, the estimate \rf{compute-1111} follows.
\end{proof}
\begin{corollary}%{lemma}%{lem}
    \lb{cor4444}
Assume that $\eta\in \Zs^{-\kp}(\Tor)$, $\kp\in (0,\frac12)$,
and let $\eta^\eps = \eta \ast \rho_\eps$.
Then, for all $\eps\gt 0$ and $j=1,2,3$,
\aaa{
\lb{kw1111}
\sup_{\substack{x\in \Tr^d, \la\in (0,1], \\ \ffi,\psi\in \td{\ms B}^1}}
\frac{\big|\big[(K \ast \v_j^\eps  - K \ast \v_j^\eps(x))\eta^\eps\big] 
\big((\ffi\psi)^\la_x\big)\big|}{\la^{1-\kp}}
\ls &\, | K\ast \v^\eps_j|_{1-\frac\kp4} |\eta^\eps|_{-\frac\kp4} \\ \ls &\, 
|\v_j|_{-1-\frac\kp4} |\eta|_{-\frac\kp4},
\nonumber 
}
where $\v_j^0: = \v_j$, $\eta^0: = \eta$.
 \end{corollary}%{lem}
\begin{proof}
Since $\v_j\in \Zs^{-1-\kp}(\Tor)$ for all $\kp\in (0,\frac12)$,
by Lemma \ref{laplace-1111}, $K\ast \v_j\in \Zs^{1-\kp}(\Tor)$. 
We apply Proposition \ref{lem4444-evelina} with $\frac\kp4$ in place of
$\kp$ and $\nu =1-\frac\kp4$.
The left-hand side of \rf{kw1111}
is not greater than $| K\ast \v^\eps_j|_{1-\frac\kp4} |\eta^\eps|_{1-\frac\kp4}$.
By Lemma \ref{lem22-22}, the latter expression is not greater than
$| K\ast \v_j|_{1-\frac\kp4} |\eta|_{1-\frac\kp4}$.
\end{proof}

\begin{proposition}%{pro}
\lb{prop-2222}
Let $\tet$ be the limit of $\tet^\eps:= \xi^\eps \x \ovl{\xi^\eps}$ in 
$\ms D'(\Tor)$ whose existence is obtained in Lemma \ref{lc1111}. Then,
for every $\kp\in (0,\frac12)$,
\aa{
\tet\in E_{-\kp}.
}
Furthermore, for all $\eps \gt 0$, $\kp\in (0,\frac12)$, and for $j=1,2,3$,
\aaa{
\lb{tet-j-1111}
&\|\tet^\eps_j\|_{-\kp} \ls |\v_{j_1}|_{-1-\frac{\kp}4} |\v_{j_2}|_{-1-\frac{\kp}4}, \\
\lb{tet1111}
 &\|\tet_j^\eps - \tet_j\|_{-\kp}  \ls 
 |\v^\eps_{j_1}-\v_{j_1}|_{-1-\frac{\kp}4} |\v_{j_2}|_{-1-\frac{\kp}4} 
 + |\v^\eps_{j_2}-\v_{j_2}|_{-1-\frac{\kp}4} |\v_{j_1}|_{-1-\frac{\kp}4},
}
where $j_1: = j+1 \mod 3$ and $j_2: = j+2 \mod 3$, 
$\tet^0: =\tet$.
\end{proposition}%{pro}
\begin{remark}
\rm
Let us observe that  \rf{tet-j-1111} implies the bound
\aa{
\|\tet^\eps\|_{-\kp} \ls  \|\fv\|^2_{-1-\frac{\kp}4}.
}
\end{remark}
\begin{proof}[Proof of Proposition \ref{prop-2222}]
By Lemma \ref{lc1111}, $\tet^\eps$ has the representation \rf{lcxxxx} 
valid for any $x\in\Tor$ and its limit $\tet$ has the representation 
\rf{lcxxxx-lim}. 

%Take $j=1,2,3$ and set $j_1: = j+1 \mod 3$ and $j_2: = j+2 \mod 3$.
In what follows, we will compute the norm $|\tet_1|_{-\kp}$.
The norms $|\tet_2|_{-\kp}$ and $|\tet_3|_{-\kp}$ can be computed likewise.
%We will show that $\tet$ is actually in $\Zs^{-\kp}$.
We will use the equivalence of the norms
$|\cdot|_{-\kp}$ and $|\cdot|_{\td{\ms Z}^{-\kp}}$,
obtained in Appendix \ref{secHZ-1111}, where the latter norm
is introduced in Definition \ref{def-h-1111}.
Furthermore, by Lemma \ref{lem:12.4}, a H\"older-Zygmund norm 
can be evaluated using an arbitrary 
test function $\varphi \in \mathcal{D}(B_1(0))$ with $\int \varphi \neq 0$ 
by inequality \rf{eq:12.3}. Take a test function of the form $\ffi^2(x)$, 
where $\ffi \in \td{\ms B}^2$. 
%Note that $\frac12 \pl_i (\ffi^2) = \ffi \pl_i\ffi$ where  $\ffi, \pl_i\ffi\in \td{\ms B}^1$.
By Lemmata  \ref{lc1111}, \ref{h-e-1111}, and \ref{lem:12.4},
for any $\kp\in (0,\frac12)$ and $\eps\gt 0$,
\eqn{
\lb{first-term}
 |\tet_1^\eps|_{-\kp} \ls &\, |\tet_1^\eps|_{\td{\ms Z}^{-\kp}} 
 %=  \sup_{\substack{x\in\Tor,\, \la \in(0,1], \\ \psi\in \td{\ms B}^1}}
%\frac{|\tet_1^\eps(\psi^\la_x)|}{\la^{-\kp}}
 \ls \sup_{x\in\Tor,\, \la \in(0,1]} \frac{|\tet_1^\eps\big((\ffi^2)^\la_x\big)|}{\la^{-\kp}}
\\
&\, =  2\sup_{x\in\Tor,\, \la \in(0,1]} 
\la^\kp\Big|\pl_1 \big[\big(K \ast \ww^\eps  - K \ast \ww^\eps(x)\big) \xi^\eps_{23}\big]
\big((\ffi^2)^\la_x\big)\\
 &\hspace{2,7cm}-\pl_2\big[\big(K \ast \ww^\eps  
 - K \ast \ww^\eps(x)\big)  \xi^\eps_{13}\big]\big((\ffi^2)^\la_x\big)\Big|  \\
&\,\lt 4 \sup_{x\in\Tor,\, \la \in(0,1]} 
\frac{\big|\big[\big(K \ast \ww^\eps  - K \ast \ww^\eps(x)\big) 
\xi^\eps_{23}\big](\ffi\pl_1\ffi)^\la_x\big|}{\la^{1-\kp}}\\
&\,+ 4\sup_{x\in\Tor,\, \la \in(0,1]} 
\frac{\big|\big[\big(K \ast \ww^\eps  - K \ast \ww^\eps(x)\big)  
\xi^\eps_{13}\big] (\ffi\pl_2\ffi)^\la_x\big|}{\la^{1-\kp}}.
}
We evaluate the first term on the right-hand side of formula  \rf{first-term}. 
The second  term can be evaluated likewise.
Since $\ffi$ and $\pl_1\ffi$ are from $\td{\ms B}^1$, we can apply 
Corollary \ref{cor4444} to obtain
\aa{
&\sup_{x\in\Tor,\, \la \in(0,1]} 
\frac{\big|\big[\big(K \ast \ww^\eps  - K \ast \ww^\eps(x)\big)  
\xi^\eps_{23}\big] \big(\ffi\pl_1\ffi\big)^\la_x\big|}{\la^{1-\kp}}\\
%&\hspace{6mm}\lt
%\sup_{\substack{x\in\Tor,\, \la \in(0,1], \\ \psi\in \td{\ms B}^1}}
%\frac{\big|\big[\big(K \ast \ww^\eps  - K \ast \ww^\eps(x)\big)  
%\xi^\eps_{23}\big] \psi^\la_x\big|}{\la^{1-\kp}}\\
&\hspace{6mm} \ls |\ww^\eps|_{-1-\frac{\kp}4} |\xi^\eps_{23}|_{-\frac{\kp}4}
%\lt | K\ast \w|_{1-\frac{\kp}2} |\xi_{22}|_{-\frac{\kp}2} 
\ls |\ww|_{-1-\frac{\kp}4} |\www|_{-1-\frac{\kp}4}.
}
This proves \rf{tet-j-1111}.
To show \rf{tet1111}, we note that by \rf{lcxxxx} and \rf{lcxxxx-lim},
\aaa{
\lb{tet-c-1111}
\tet_1^\eps - \tet_1 = &\, 2 \pl_1\big[\big(K\ast (\ww^\eps - \ww)   
- K \ast (\ww^\eps(x) - \ww(x))\big)\xi^\eps_{23} \big]\\
\lb{tet-c-2222}
- &\, 2\pl_2 \big[\big(K\ast (\ww^\eps - \ww)   
- K \ast (\ww^\eps(x) - \ww(x))\big)\xi^\eps_{13} \big]\\
\lb{tet-c-3333}
+&\, 2\pl_1\big[(K\ast \ww - K\ast \ww(x))(\xi^\eps_{23}- \xi_{23})\big] \\
\lb{tet-c-4444}
- &\, 2\pl_2\big[(K\ast \ww - K\ast \ww(x))(\xi^\eps_{13}- \xi_{13})\big]\\
= : &\, \te_1 + \te_2,\notag
}
where $\te_1$ is the sum of the term on the right-hand side 
of \rf{tet-c-1111} and the term in \rf{tet-c-2222};
$\te_2$ is the sum of the terms in \rf{tet-c-3333} and \rf{tet-c-4444}.
From (\ref{tet-c-1111}--\ref{tet-c-4444}) 
it follows that the argument on obtaining the bound for  $|\tet_1^\eps|_{-\kp}$ works
for both $\te_1$ and $\te_2$. In the case of $\te_1$, in the above argument,
 we substitute $\ww^\eps$ with $\ww^\eps - \ww$ to obtain the bound
\aa{
|\te_1|_{-\kp} \ls  |\ww^\eps - \ww|_{-1-\frac{\kp}4} \big(|\xi^\eps_{13}|_{-\frac{\kp}4} 
+ |\xi^\eps_{13}|_{-\frac{\kp}4} \big)\ls  |\ww^\eps - \ww|_{-1-\frac{\kp}4} |\www|_{-1-\frac{\kp}4}.
}
In the case of $\te_2$, in the argument on obtaining the bound for  
$|\tet_1^\eps|_{-\kp}$, 
we substitute $\xi^\eps_{23}$ with $\xi^\eps_{23} - \xi_{23}$ 
and $\xi^\eps_{13}$ with $\xi^\eps_{13} - \xi_{13}$. Thus, we obtain 
\aa{
|\te_2|_{-\kp} \ls | \ww^\eps|_{-1-\frac{\kp}4} |\xi^\eps_{23}-\xi_{23}|_{-\frac{\kp}4}
+ | \ww^\eps|_{-1-\frac{\kp}4} |\xi^\eps_{13}-\xi_{13}|_{-\frac{\kp}4}
\ls   |\ww|_{-1-\frac{\kp}4} |\www^\eps - \www|_{-1-\frac{\kp}4}.
}
This implies \rf{tet1111}.
\end{proof}
We have the following corollary which implies Lemma  \ref{zeta-lemma-1111}.
\begin{corollary}%{pro}
\lb{cor-2222}
Let $\zeta$ be the limit of $\zeta^\eps:= 2 G \ast (\xi^\eps \x \ovl{\xi^\eps})$ in 
$\ms D'(\Tor)$. Then,
\aa{
\zeta\in E_{1-\kp}.
}
Furthermore, for all $\eps  \gt 0$ and $\kp\in (0,\frac12)$, 
\aa{
%\lb{zeta1111}
&\|\zeta^\eps\|_{1-\kp} \ls \|\fv\|^2_{-1-\frac{\kp}4}, \\
%\lb{zeta2222}
 &\|\zeta_j^\eps - \zeta_j\|_{1-\kp}  \ls 
 |\v^\eps_{j_1}-\v_{j_1}|_{-1-\frac{\kp}4} |\v_{j_2}|_{-1-\frac{\kp}4} 
 + |\v^\eps_{j_2}-\v_{j_2}|_{-1-\frac{\kp}4} |\v_{j_1}|_{-1-\frac{\kp}4},
}
where  $j=1,2,3$, $j_1: = j+1 \mod 3$, $j_2: = j+2 \mod 3$,
$\zeta^0: =\zeta$.
\end{corollary}%{pro}

\section{Rate of convergence and proof of the main result}
\label{sec-rate of converegnce}
The aim of this section is to prove Theorem \ref{thm-main-1111}. This is done at the end of the section.

The next result is important for determining the rate of convergence
\rf{r-rate}. More specifically, it is used to 
prove the convergence $R^\eps \to R$
in $E_{1-\kp}$ and evaluate the rate of this convergence.
\begin{lemma}%{lem}
\lb{l12-222}
The following inequalities are satisfied for all $\eps>0$
and $\kp\in (0,\frac12):$
\aaa{
\lb{2nd-tet-1111}
&\|\xi^\eps-\xi\|_{-\kp}  \ls \eps^\frac{\kp}2 \|\fv\|^2_{-1-\frac{\kp}2} ,\\
%\sum\nolimits_{j=1}^3 |\v_j|^2_{-1-\frac{\kp}4}
&\|\tet^\eps - \tet\|_{-\kp} \ls \eps^{\frac{\kp}8}  \|\fv\|^2_{-1-\frac{\kp}4},\\
\lb{3rd-tet-1111}
&\|\zeta^\eps - \zeta\|_{1-\kp} \ls  \eps^{\frac{\kp}8}  \|\fv\|^2_{-1-\frac{\kp}4}.
}
\end{lemma}%{lem}
\begin{proof}
By Lemma \ref{wn1111},
\aa{
|\v_j^\eps - \v_j|_{-1-\kp} \ls \eps^{\frac{\kp}2} |\v_j|_{-1-\frac{\kp}2},
\quad j=1,2,3, \quad \text{a.s.}
}
Since $\xi_{kj} = 2 G_k \ast \v_j$, we have
\aa{
|\xi^\eps_{kj} - \xi_{kj}|_{-\kp} =  2|G_k \ast \v_j^\eps - G_k \ast \v_j|_{-\kp}
\ls   |\v_j^\eps - \v_j|_{-1-\kp} \ls \eps^\frac{\kp}2|\v_j|_{-1-\frac{\kp}{2}}.
}
Further, by Proposition \ref{prop-2222},
\aa{
 |\tet^\eps_1 - \tet_1|_{-\kp} \ls |\ww^\eps-\ww|_{-1-\frac{\kp}4} |\www^\eps|_{-1-\frac{\kp}4}
 + |\www^\eps-\www|_{-1-\frac{\kp}4} |\ww|_{-1-\frac{\kp}4} \\
 \ls \eps^{\frac{\kp}8} |\ww|_{-1-\frac{\kp}4} |\www|_{-1-\frac{\kp}4}.
 }
 Likewise, we evaluate $|\tet^\eps_j - \tet_j|_{-\kp}$  for $j=2,3$.
 This implies \rf{2nd-tet-1111}. Finally, recalling that
 $\zeta^\eps = 2G \ast \tet^\eps$, we obtain \rf{3rd-tet-1111}.
% -2G \ast \big(R\x \ovl R+ (\gm  \xi^\eps)\x \ovl R + R\x (\gm \ovl {\xi^\eps})\big) 
%+ 2i\,G \, \ast A_\sg R + 2i\, G \ast \gm \xi^\eps - \td \gm \zeta^\eps
 \end{proof}
%\begin{lemma}%{lem}
%For each $(R,S) \in M_{1-\kp,\sg}$,
%\aa{
%\lim_{\eps\to 0} \Gm_\eps(R,S) = \Gm(R,S) \quad  \text{in}\;\;  E_{1-\kp}.
%}
%\end{lemma}%{lem}
%\begin{proof}
%Since
%\aa{
%\lim_{\eps\to 0} |\xi^\eps_{ij} - \xi_{ij}|_{-\kp} = 0 \quad \text{and} \quad
%\lim_{\eps\to 0} |\zeta^\eps - \zeta|_{1-\kp} = 0.
%}
%for $i,j=1,2$, the statement follows.
%\end{proof}
\begin{lemma}%{lem}
\lb{rr-2222}
Let $\dl>0$, $\kp\in (0,\frac12)$, and let $\Om_\dl\sub\Om$ be the set constructed 
in Proposition \ref{pro2222}. Further let $\sg,\sig>0$ be numbers,
defined in the same proposition, for which the maps
$\Gm_{\eps,\sg}$ and $\Gm_{0,\sg}$ possess unique fixed points 
$R^\eps$ and, respectively,  $R$ in $M_{\sg,1-\kp}$  for 
$|\gm|<\sig$. Then, 
there exists a constant $C_\delta >0$, depending only on $\dl$, such that
for all $\om\in \Om_\dl$,
\aaa{
\lb{RRSS1111}
\|R - R^\eps\|_{1-\kp} \lt C_\dl\, \eps^\frac{\kp}8. 
}
% where the constant on the 
% In other words, there exists $C(\delta)>0$ such that 
% \aaa{
% \lb{RRSS1111-ZB}
% \|R(\omega) - R^\eps(\omega)\|_{1-\kp} \lt C(\delta)\eps^\frac{\kp}2, \;\; \omega \in \Omega_\delta. 
% }
\end{lemma}%{lem}
\begin{proof}Let us choose and fix $\eps>0$. Let us recall that $R^\eps$ is a solution of  equation \rf{R-eps-1111}  and $R$ 
is a solution of  equation \rf{R-eps-1111}  with  $\eps=0$.
 We write  $\xi$ and $\zeta$ instead of 
 $\xi^0$ and $\zeta^0$ when $\eps = 0$. Similar to \rf{ineq1111}, 
keeping the same constants $C$ and $C_\La$ as in \rf{ineq1111}, using equation \rf{R-eps-1111} 
and recalling that $|\gm|<\frac{\sg^2}{16\La}$, we infer that there exists  $\td C>0$ such that 
\mm{
\|R - R^\eps\|_{1-\kp} \lt
2C\big( \sg + |\gm|\|\xi^\eps\|_{-\kp}  \big)
\|R-R^\eps\|_{1-\kp} + 2C\sg |\gm|\|\xi-\xi^\eps\|_{-\kp}  
+ |\gm|^2\|\zeta-\zeta^\eps\|_{1-\kp},\\
\lt 2C_\La  (\sg+\sg^2)  \|R-R^\eps\|_{1-\kp} + \td C \big( \|\xi-\xi^\eps\|_{-\kp}
+\|\zeta-\zeta^\eps\|_{1-\kp}\big).
}

Recall that $C_\La$ satisfies \rf{fp-1111} so that $2C_\La  (\sg+\sg^2)  <1$.
By Lemma \ref{l12-222} and the choice \rf{omdl1111} of the set $\Om_\dl$,
we obtain \rf{RRSS1111}.
\end{proof}

In order to conclude the proof of Theorem \ref{thm-main-1111} we finally need 
the following result on the rate of convergence of $r^\eps$. It uses the change of variables  \rf{Rxi2222}.

\begin{proposition}
\lb{pro-2222}
Let the objects $\dl>0$, $\kp\in (0,\frac12)$, $\Om_\dl\sub\Om$, $\sig>0$, $R^\eps$, and $R$
be as in Lemma \ref{rr-2222}.
% Let $\dl>0$ and let $\Om_\dl\sub\Om$ be the set constructed 
% in Proposition \ref{pro2222}. Further let $\sg,\sig>0$ be numbers,
% defined in the same proposition, for which the maps
% $\Gm_{\eps,\sg}$ and $\Gm_{0,\sg}$ possess unique fixed points 
% $R^\eps$ and $R$, respectively, for 
% $|\gm|<\sig$. 
Then,  for every $\eps>0$ and $\om\in \Om_\dl$, the function \[r^\eps = R^\eps - \gm \xi^\eps,\]
where $|\gm|<\sig$, solves \rf{sCR}.
Moreover, there exists a limit $r: = \lim_{\eps\to 0} r^\eps$ in $E_{-\kp}$
and a constant $\hat C_\delta >0$, depending only on $\dl$, such that
for all $\om\in \Om_\dl$,
\aaa{
\lb{r-eps-1111}
\|r^\eps-r\|_{-\kp} \lt \hat C_\delta\, \eps^{\frac\kp8}.
}
\end{proposition}
\begin{proof}
By Lemma  \ref{lem-R-3333},
 \rf{sys_r1111} and \rf{R-eps-1111} are equivalent for every $\eps>0$
 (provided that $a,b,c$  are as in the aforementioned lemma).
 Since for all $\eps > 0$, $R^\eps\in E_{1-\kp}$  and 
 $\xi^\eps\in E_{1-\kp}$ (see Remark \ref{remark-3.5}), 
by Proposition \ref{pro-123-1111}, $r^\eps$ solves \rf{sCR}. 
Define 
\aa{
r: = R - \gm \xi.
}
Then, since $|\v_j|_{-1-\kp}$, $j=1,2,3$, are bounded in $\om\in\Om_\dl$,
 \rf{r-eps-1111} is fulfilled on $\Om_\dl$ by Lemmata \ref{l12-222} and \ref{rr-2222}.
\end{proof}

\begin{proof}[Proof of Theorem \ref{thm-main-1111}]
Let the objects $\dl>0$, $\Om_\dl\sub\Om$, $\sig>0$, $r^\eps$, and $r$
be as in Proposition \ref{pro-2222} and $|\gm|<\sig$. By Proposition \ref{pro-2222},
$r^\eps$ solves \rf{sCR} and there exists a limit $r: = \lim_{\eps\to 0} r^\eps$ in $E_{-\kp}$. By Lemma \ref{fh2222},
$\ms Z^{-\kp}(\Tor)$ and $\C^{-\kp}(\Tor)$ coincide for $\kp\in (0,\frac12)$. Therefore,
$r^\eps\to r$ also in $\C^{-\kp}(\Tor)$. The bound \rf{r-eps-1111} 
is therefore exactly what is required in \rf{r-rate}. The theorem is proved.
\end{proof}

%\section{A white noise perturbed constant mean curvature equation}
\section{B\"acklund-transformed,
white-noise perturbed Landau-Lifshitz-Gilbert equation}
\lb{c-mean-1111}
In this section we prove the existence of solution to equation \rf{B-L}.
Here, we additionally assume that 
$\w$ and $\ww$ have  zero means over $\Tor$.
 \begin{theorem}
Let $\fv = (\w,\ww,\www)$ be a 3D white noise on $\Tor$ 
with the zero mean over $\Tor$ and let $r$ be the solution to \rf{CR} constructed in 
Theorem \ref{thm-main-1111}. Then,
\aaa{
\lb{B-r}
B = 2i \, \ovl G \ast r
}
belongs to $E_{1-\kp}$ and is a solution to \rf{B-L}. Moreover, 
 $\Im B = 0$.
%and $B(x,y)$, $(x,y)\in\Tor$, is real-valued.
\end{theorem}
\begin{proof}
% Note that in \rf{sys_r1111}, $a=b=0$ by \rf{ab-1111}. We can also set $c=0$.
% Indeed, the role of $a,b$, and $c$ was to recover $\eta_1$ and $\eta_2$.
Note that since $G = 2\pl_z K$, by \rf{sys_r1111}, $\int_\Tor r^\eps = 0$.
It is straightforward to verify that if we set
\aaa{
\lb{rr-1111}
r^\eps = \frac12(\pl_y B^\eps + i\pl_x B^\eps) = i \pl_z B^\eps,
}
then, since $\pl_{\bar z} r^\eps = \frac{i}4 \lap B^\eps$
and $r^\eps \x \ovl{r^\eps} = \frac{i}2 (\pl_x B^\eps \x \pl_y B^\eps)$,
we obtain that $B^\eps$ satisfies the equation
\aa{
\lap B^\eps =2 \pl_x B^\eps \x \pl_y B^\eps + 4\gm\fv^\eps,
}
where $\fv^\eps = \fv\ast \rho_\eps$.
By \rf{rr-1111} and Lemma \ref{Lav1111}, 
$B^\eps = 2i\, \ovl G \ast r^\eps$ is a solution.
Passing to the limit as $\eps\to 0$, we obtain \rf{B-r}. Note that
since $r\in E_{-\kp}$ and $i\ovl G = 2i \pl_{\bar z} K$, 
$B\in E_{1-\kp}$. 

Finally,  by \rf{sCR}, $\Re (\pl_{\bar z} r^\eps) = 0$,
which is equivalent to $\pl_x p^\eps = \pl_y q^\eps$, where $p^\eps = \Re r^\eps$
and $q^\eps = \Im r^\eps$. On the other hand,
\aa{
\Im (i \ovl G \ast r^\eps) =   \pl_x K \ast p^\eps - \pl_y K \ast q^\eps
= K\ast (\pl_x p^\eps - \pl_y q^\eps) = 0
}
which proves that $\Im B=0$.
\end{proof}

In the appendices, we collect some important useful results.

\appendix

%\section{Appendix}
 
\section{On the Cauchy-Riemann operator}
\begin{lemma}
\lb{Lav1111}
Let $G$ be the Green function  defined by \rf{green-1111}.
If $f\in\C^\al(\Tor)$, $\al\in (0,1)$, $C\in\Cnu$ is  a constant and 
a function $r:\Tor \to \mathbb{C}$ satisfies 
%If
%\coma{assumptions about $f$}
\aaa{
\lb{rG1111}
r = G \ast f + C
}
then $r\in \C^{1+\al}(\Tor)$ and satisfies the identity 
\aa{
-2\pl_{\bar z} r = f - [f],
}
where $[f]: = (4\pi^2)^{-1}\int f$ denotes
the average over $\Tor$. %In particular, $r\in \C^{1+\al}(\Tor)$.
\end{lemma}
\begin{proof}
By Lemmata \ref{fh2222} and \ref{Gconv-1111}, $r\in \C^{1+\al}(\Tor)$.
Note that 
\begin{equation}\label{eqn-CR^2=Delta}
    \pl_z\pl_{\bar z} = \frac14 \lap \mbox{ on }\C^{2+\al}(\Tor). 
    \end{equation}
Thus, by  applying  the operator $2\pl_{\bar z}$ to both sides of equality 
\rf{rG1111} we infer that 
\aa{
2\pl_{\bar z} r = 2\pl_{\bar z} G \ast f
= 4   \pl_{\bar z} \pl_z K\ast  (f-[f])=
\lap (K \ast (f-[f])) = [f] - f.
}
To justify the last step, we observe that if $\int_\Tor g =0$, then
\begin{align}
-\lap u = g \quad  \iff \quad
u =K \ast g + C.
\end{align}
for some constant $C\in \Rnu$. 
\end{proof}
\section{On the white noise on $\Tr^d$, $d\gt 2$}
\begin{lemma}
\lb{noise-f-2222}
Let $\fv$ be a 3D Gaussian white noise, as in Definition
\ref{white-noise}.
If $\v^\eps: = \v_j \ast \rho_\eps$, $j=1,2,3$, are mollified white noises,
where $\rho_\eps$ is a standard mollifier, then,
\aaa{
\lb{v-noise-f-1111}
\v_j^\eps = \eta_j+
 \rho_\eps \ast\Re\Big[\sum_{k\in \Znu^d\setminus \{0\}} \eta_{kj} e^{i(k,x)}\Big].
}
\end{lemma}
\begin{proof}
Applying the convolution with $\rho_\eps$ in \rf{FS-2222}, we obtain \rf{v-noise-f-1111}.
\end{proof}
\begin{lemma}%{lem}
\lb{pro1111}
Let $\v: \Om \to \ms D'(\Tr^d)$ be a Gaussian white noise.
Then, for any $\kp>0$, $\v \in \C^{-\frac{d}2-\kp}(\Tr^d)$ a.s.
Moreover, for any $\dl>0$, there exists $\La>0$ such that
\aa{
\PP(\|\v\|_{\C^{-\frac{d}2-\kp}(\Tr^d)}<\La)>1-\dl.
}
\end{lemma}%{lem}
\begin{proof}
The first part of the Lemma is known, see e.g., \cite{bcz2023-2}  or \cite{m.veraar2011}, according to which,   for any $\kp>0$,
\aa{
\PP\big\{\v \in  \C^{-\frac{d}2-\kp}(\Tr^d)\big\} = 1.
}
The second part trivially follows from the first but since we use it, we show a  proof. 
\aa{
\big\{\v \in  \C^{-\frac{d}2-\kp}(\Tr^d)\big\}  = \bigcup_{N=1}^\infty A_N,
\quad \text{where} \quad
A_N = \big\{\|\v\|_{\C^{-\frac{d}2-\kp}(\Tr^d)} < N\big\}.
}
Since $A_N \sub A_{N+1}$, by the continuity of the measure $\PP$,
\aa{
\lim_{N\to \infty} \PP\big\{\|\v\|_{\C^{-\frac{d}2-\kp}(\Tr^d)} < N\big\} = 1.
}
\end{proof}
\begin{lemma}%{lem}
\lb{wn1111}
Let $\v: \Om \to \ms D'(\Tr^d)$ be a Gaussian white noise.
Then, for any $\kp > 0$, $\v \in \Zs^{-\frac{d}2-\kp}(\Tr^d)$ a.s.
Furthermore, for any  $\eps>0$
% \aaa{
% \lb{est1111}
% |\v^\eps - \v|_{-1-\kp} \ls \eps^{\frac{\kp}2} |\v|_{-1-\frac{\kp}2},
%  \quad \text{a.s.}
% }
and $\kp>\kap>0$, 
\aaa{
\lb{est1111}
|\v^\eps - \v|_{-\frac{d}2-\kp} \ls \eps^{\kp-\kap}
|\v|_{-\frac{d}2-\kap} \, ,
 \quad \text{a.s.}
}
\end{lemma}%{lem}
%For the proof of Lemma \ref{wn1111}, we need the following proposition.
\begin{proof} 
 Let us choose and fix $\kp>\kap>0$ and
let $\ell = [\frac{d}2+\kp +1]$, where $[\fdot]$ denotes
the integer part.
By Lemma \ref{pro1111},
$\v \in \ms{Z}^{-\frac{d}2-\kap} \cap \ms{Z}^{-\frac{d}2-\kp}$. 
By Proposition \ref{eq1111}, one can compute 
$|\v^\eps - \v|_{\ms Z^{-\frac{d}2-\kp}}$
by taking the supremum over  $\ffi\in \ms B^{\ell+1}$. 
First, we note that
\aa{
\v^\eps(\ffi^\la_x) = \v(\rho_\eps \ast \ffi^\la_x)
= \v\Big(\int \rho(z)\ffi^\la_x(\fdot - \eps z)dz\Big) = 
\int \rho(z)\v(\ffi^\la_{x+\eps z}) dz,
}
where $\rho$ is a standard mollifier. Define
$\psi: \Tr^d \x \Rnu^d \to \Rnu$,
\aa{
\psi(y;w): = \frac{\ffi(y - w) - \ffi(y)}{|w|^{\kp-\kap}}.
}
We have
\aaa{
|\v^\eps - \v|_{\ms Z^{-\frac{d}2-\kp}} 
\ls &\sup_{\substack{x\in\Tr^d\!,\, \la\in [0,1]_\D, \\ \ffi\in \ms B^{\ell+1}}}
\frac{|(\v^\eps - \v)(\ffi^\la_x)|}{\la^{-\frac{d}2-\kp}}\notag\\
\lt &\int \rho(z)  dz
 \sup_{\substack{x\in\Tr^d\!,\, \la\in [0,1]_\D, \\ \ffi\in \ms B^{\ell+1}}}
 \frac{|\v(\ffi^\la_{x+\eps z} - \ffi^\la_x)|}{\la^{-\frac{d}2-\kp}}\notag\\
\lt &\,\eps^{\kp-\kap} \int \rho(z)|z|^{\kp-\kap}  dz 
 \sup_{\substack{x\in\Tr^d\!,\, \la\in [0,1]_\D, \\ \ffi\in \ms B^{\ell+1}}}
 \frac{|\v\big(\psi^\la_x(\fdot; \frac{\eps z}\la)\big)|}{\la^{-\frac{d}2-\kap}} \notag\\
 \ls &\,\eps^{\kp-\kap}  
 \Big[\,
 \sup_{\substack{x\in\Tr^d\!,\, \la\in [0,1]_\D, \\ \ffi\in \ms B^{\ell+1}\!,
 \, |w|\lt 1}}
 \frac{|\v(\psi^\la_x(\fdot; w))|}{\la^{-\frac{d}2-\kap}}
 +
  \sup_{\substack{x\in\Tr^d\!,\, \la\in [0,1]_\D, \\ \ffi\in \ms B^{\ell+1}
  \!,\, |w| > 1}}
 \frac{|\v(\psi^\la_x(\fdot; w))|}{\la^{-\frac{d}2-\kap}}\,
  \Big].
  \lb{bound1111}
}
Since for $\ffi\in\ms B^{\ell+1}$, $\|D^l\ffi\|_\infty\lt 1$, $|l| = 1,2 \ldots \ell+1$, we have the following bound uniformly in $|w|\lt 1$:
\aaa{
\lb{bb-3333}
%\Big\|\psi\Big(\fdot; \frac{\eps z}\la\Big)\Big\|_\infty \ls 1 \quad \text{and} \quad
\big\|\pl^k_1 \psi(\fdot; w)\big\|_\infty \ls
 \|D^l\ffi\|_\infty \ls 1, \quad |l| = |k|+1, \quad |k| = 0,1,\ldots \ell,
}
where by $\pl^k_1$, we understand a partial derivative of order $|k|$ with respect to
the first argument of $\psi$ and by $D^l$ we understand a derivative of order $|l|$
of $\ffi$. 
By \rf{bb-3333}, 
the first supremum in the line \rf{bound1111} is bounded by
\aa{
\sup_{\substack{x\in\Tr^d\!,\, \la\in [0,1]_\D, \\ \psi\in \ms B^\ell }}
\frac{|\v(\psi^\la_x)|}{\la^{-\frac{d}2-\kap}}
\ls |\v|_{\ms Z^{-\frac{d}2-\kap}},
}
where the last bound is implied by Proposition \ref{eq1111},
as $\ell\gt [\frac{d}2+\kap+1]$.
The second supremum in the line \rf{bound1111} is bounded by
\aa{
\sup_{\substack{x\in\Tr^d\!,\, \la\in [0,1]_\D, 
\\ \ffi\in \ms B^{\ell+1} \!,\, |w| > 1}}
 \frac{|\v((\ffi_{-w})^\la_x)|}{\la^{-\frac{d}2-\kap}}
 + \sup_{\substack{x\in\Tr^d\!,\, \la\in [0,1]_\D, \\ \ffi\in \ms B^{\ell+1}}}
 \frac{|\v(\ffi^\la_x)|}{\la^{-\frac{d}2-\kap}} \ls
 |\v|_{\ms Z^{-\frac{d}2-\kap}}.
}
This immediately implies \rf{est1111}.
%\aa{
%|\w^\eps - \w|_{-1-\kp} \ls \eps^{\frac{\kp}2} |\w|_{-1-\frac{\kp}2}.
%}
\end{proof}

\section{On H\"older-Zygmund spaces and norms}
\lb{secHZ-1111}
In  \cite[Definition A.9]{bcz2023-1}, the sets  $\ms B^r$, $\ms B_\al$,
and $\ms B^r_\al$ were defined in a different manner. Here we would like
to show the equivalence of the corresponding  H\"older-Zygmund norms.
%To do so, we introduce
%the following sets of test functions. 
\begin{remark}
\lb{remark-ball-1111}
Let $B_1(0)\sub\Rnu^d$ be an open ball of radius 1 centered at 
the origin.
If $\ffi\in \ms D(B_1(0))$, then $\ffi_0^\la\in \ms D(B_1(0))$ for $\la\in (0,1)$.
Since $B_1(0)\sub [-\pi,\pi]^d$, both $\ffi$ and  $\ffi_0^\la$
can be extended $2\pi$-periodically, in both variables, from $[-\pi,\pi]^d$ to $\Rnu^d$. 
This way, $\ffi_y^\la = \ffi_0^\la(\fdot - y)$,  $\la\in (0,1]$,  can be regarded as a function from $\ms D(\Tr^d)$. 
\end{remark}
% \begin{remark}
% \lb{remark-ball-2222}
% Let $B_1(x): = B_1(0) +x$, $x\in\Rnu^d$.
% If $\ffi\in  \ms D(B_1(x))$, then $\ffi = \psi(\cdot+x)$
% for some function $\psi\in \ms D(B_1(0))$ (namely, $\psi = \ffi(\cdot -x)$).
% Since, by Remark \ref{remark-ball-1111}, 
% $\psi$ can be regarded as function from $\ms D(\Tr^d)$,
% so can be $\ffi$.
% \end{remark}

\begin{definition}
\lb{def-h-1111}
For $r\in\Nnu \cup \{0\}$ and $\al\in\Rnu$,  we define
\aa{
&\td {\ms B}^r: = \big\{\ffi\in\mc D(B_1(0)): \|\pl^k\ffi\|_\infty\lt 1, \; 
\text{for all}\; 0\lt |k| \lt r\big\},\\
%\label{eqn-B_a}
&\td {\ms B}_\al: = \big\{\ffi\in\mc D(B_1(0)):
\int_{\Tr^d} \ffi(z) z^k dz = 0, \; \text{for all}\; 0\lt |k| \lt \al\big\},
\, \mbox{ if } \al \gt 0,\\
%\label{eqn-B^r_a}
&\td {\ms B}^r_\al: = \td {\ms B}^r \cap \td {\ms B}_\al,
\, \mbox{ if } \al \gt 0,
}
provided that $\ms D(B_1(0))\sub \ms D(\Tr^d)$
in the sense of Remark \ref{remark-ball-1111}.
\end{definition}
\begin{definition}
\lb{C-3}
Define the H\"older-Zygmund norms 
% defined with respect to the family
% of test functions from the sets $\ms B^{[-\al+1]}(x)$ ($\al<0$) 
% and $\ms B^0_\al(x)$ ($\al\gt 0$), we denote by 
$|\cdot|_{\ms Z^\al\!,\, x}$ as follows.
If $ \al\gt 0$, then
\aaa{
\lb{eqn-norm-a,x-1}
&|f|_{\ms Z^\al\!,\, x}:=
\sup_{y\in\Tr^d\!, \, \psi_x \in \td{\ms B}^0} |f(\psi_y)|
+ \sup_{\substack{y\in\Tr^d\!, \,\la\in (0,1],\\ \ffi_x\in 
\td{\ms B}^0_\al}}\frac{|f(\ffi^\la_y)|}{\la^\al}.}
If $ \al< 0$, then, with $r=[-\al+1]$,  
\aa{
|f|_{\ms Z^\al\!,\, x}
:=  \sup_{\substack{y\in\Tr^d,\, \la\in (0,1],\\ \ffi_x\in 
\td{\ms B}^r}}\frac{|f(\ffi^\la_{y})|}{\la^\al}.
}
We define the H\"older-Zygmund space $\td {\ms Z^\al}(\Tr^d)$ 
as a space of distributions $f\in \ms D'(\Tr^d)$  for which
$|f|_{\ms Z^\al\!,\, 0} < \infty$.
For simplicity of notation, we set
\aa{
|f|_{\td{\ms Z}^\al}:= |f|_{\ms Z^\al\!,\, 0}.
}
\end{definition}
\begin{remark}
\rm
\lb{c5}
Remark that $|\cdot|_{\td{\ms Z}^\al}$  is the same norm as in 
\cite[Definition A.9]{bcz2023-1}.
\end{remark}
\begin{remark}
\lb{hz-seminorm}
\rm
For $\al\in\Rnu$, the H\"older-Zygmund seminorm $|\cdot|_{\al, \mcc K}$
is defined  in exactly the same way as the norm 
$|\cdot|_{\td{\ms Z}^\al}$, see Definition
\ref{C-3}, but with $\Tr^d$ replaced with a compact $\mcc K\sub\Rnu^d$.
For the precise definition of  $|\cdot|_{\al, \mcc K}$ and the 
H\"older-Zygmund space $\ms Z^\al(\Rnu^d)$,
we refer to \cite[Definition 2.1]{bcz2023-1}.
\end{remark}
\begin{remark}
\lb{intersection-1111}
\rm
Note that $\td{\ms Z}^\al(\Tr^d) = \ms Z^\al(\Rnu^d)\cap \ms D'(\Tr^d)$.
Indeed, the norm $|\cdot|_{\td{\ms Z}^\al}$ is the restriction 
of the  seminorm $|\cdot|_{\al, \Tr^d}$ to periodic distributions. 
Conversely, each seminorm $|\cdot|_{\al, \mcc K}$, where
$\mcc K\sub\Rnu^d$ is a compact, can be bounded by  $|\cdot|_{\al, \Tr^d}$
due to the periodicity of the distributions.
\end{remark}
Let $\al\notin \Nnu$. The following lemma is about the equivalence
of the norms $|\cdot|_{\td{\ms Z}^\al}$ and $|\cdot|_{\ms Z^\al}$, 
where the latter
is the norm introduced in Definition \ref{HZ1111}. 
\begin{lemma}
\lb{h-e-1111}
For every $\al\in\Rnu$, $\al\notin \Nnu$, $\ms Z^\al=\td{\ms Z}^\al$ 
and the norms 
 $|\cdot|_{\td{\ms Z}^\al}$ and $|\cdot|_{\ms Z^\al}$ are equivalent.
\end{lemma}
\begin{proof}
First, we show that $|\cdot|_{\td{\ms Z}^\al}$ and 
$|\cdot|_{\ms Z^\al\!,\, x}$ is the same norm.
Indeed, let $\al<0$ and $r=[-\al+1]$. We have
\aa{
|f|_{\ms Z^\al\!,\, x} =
  \sup_{\substack{y\in\Tr^d,\, \la\in (0,1],\\ \ffi_x\in 
  \td{\ms B}^r}}\frac{|f(\ffi^\la_y)|}{\la^\al}=
  \sup_{\substack{y\in\Tr^d,\, \la\in (0,1],\\ \ffi\in 
  \td{\ms B}^r}}\frac{|f(\ffi^\la_{y+x})|}{\la^\al}=
  |f|_{\td{\ms Z}^\al}.
}
The latter identity holds by the $2\pi$-periodicity of $\ffi_{y+x}$ in both variables.
For $\al\gt 0$, the proof is the same.

Next, let us choose and fix a finite cover of $\Tr^d$
by open balls  $\{B_{x_k}(1)\}$. Furthermore, we choose a finite smooth partition of unity 
$\{\chi_k\}$ of $\Tr^d$ subordinated to the cover  $\{B_{x_k}(1)\}$.
Let $f\in \td{\ms Z^\al}$, $\al<0$, and $r=[-\al+1]$. Then
\eqn{
\lb{same}
 |f|_{\ms Z^\al}
 = &\, \sup_{\substack{x\in\Tr^d,\, \la\in [0,1]_\D,\\ 
 \ffi\in \ms B^r}}\frac{|f(\ffi^\la_{x})|}{\la^\al}
 \lt \sum_k \sup_{\substack{x\in\Tr^d,\, \la\in [0,1]_\D,\\ \ffi\in \ms B^r}}
 \frac{|f\big((\ffi \chi_k)^\la_{x}\big)|}{\la^\al} \\
 \ls &\, \sum_k \sup_{\substack{x\in\Tr^d,\, \la\in (0,1],\\ 
 \psi_{x_k}\in \td{\ms B}^r}}
 \frac{|f\big(\psi^\la_{x}\big)|}{\la^\al} 
 = \sum_k |f|_{\ms Z^\al\!,\,x_k}
 \ls |f|_{\td{\ms Z}^\al}.
 }
 Now let $\al\gt 0$, 
 $\al\notin \Nnu$, $n=[\al]$,  and, for a fixed $x\in\Rnu^d$, let $T_{n,x}$ be the Taylor 
polynomial of degree $n$, based at 0, for the function $f(x+\fdot)$. 
Since, the test function $\ffi$ in \rf{HZ1}
annihilates monomials  of order up to $n$, Taylor's formula 
(with Lagrange remainder)  implies
%By Definition \ref{HZ1111},
\eqn{
\lb{taylor-1111}
 \sup_{\substack{x\in\Tr^d\!, \,\la\in [0,1]_\D,\\ \ffi\in \ms B^0_\al}}\frac{|f(\ffi^\la_x)|}{\la^\al}
 = &\, \sup_{\substack{x\in\Tr^d\!, \,\la\in [0,1]_\D,\\ \ffi\in \ms B^0_\al}}
 \frac{|(f(x+\la \fdot)- T_{n,x}(\la \fdot))(\ffi)|}{\la^\al}\\
 \ls &\, \sum_{|\beta| = n}\, [D^\beta f]_\al 
 \ls \, \|f\|_{\C^\al(\Tr^d)} \ls |f|_{\td{\ms Z}^\al},
}
where $\beta$ is a multiindex and $D^\beta$ is a partial
derivative of order $|\beta|$.
 The evaluation of the first term in \rf{HZ1} can be done by the same argument
 as in \rf{same}. Thus, we have proved that
 \aa{
   |f|_{\ms Z^\al} \ls |f|_{\td{\ms Z}^\al}.
 }
 Let us prove the inverse inequality. Let $f\in \ms Z^\al$,
 $\la\in (0,1]$, and let $\la_{n+1}=2^{-(n+1)}$, $\la_n = 2^{-n}$,
 $n\in \Nnu_0$, be two consecutive points from $[0,1]_\D$
 with the property
 \aa{
 \la_{n+1} < \la \lt \la_n.
 }
 Note that $\frac{\la_n}{\la}< \frac{\la_n}{\la_{n+1}} = 2$.

 Let $\al<0$, $r=[-\al+1]$, and $\ffi\in \td{\ms B}^r$. We have
 \aa{
 \frac{|f(\ffi^\la_{x})|}{\la^\al} \lt
 \Big(\frac{\la_n}{\la}\Big)^{\!d}\,
 \frac{|f(\psi^{\la_n}_x)|}{\la_n^\al} 
 < 2^d \, \frac{|f(\psi^{\la_n}_x)|}{\la_n^\al}, 
 \quad \text{where} \;\; \psi(y) = \ffi\Big(\frac{\la_n}{\la} y\Big).
 }
 Note that since $1\lt \frac{\la_n}{\la}<2$, 
 \aa{
 \supp \psi \sub B_1(0) \sub [-\pi,\pi]^d \quad \text{and} \quad
 \|D^k \psi\|_\infty \lt 2^{|k|} \|D^k \ffi\|_\infty \lt 2^r.
 }
 Therefore,
 \aaa{
 \lb{ls-1111}
 |f|_{\td{\ms Z}^\al} \ls  |f|_{\ms Z^\al}.
 }
 Now let $\al\gt 0$ and $\ffi \in \td{\ms B}^0_\al$.
 We have
 \aa{
 \frac{|f(\ffi^\la_{x})|}{\la^\al} \lt
 \Big(\frac{\la_{n+1}}{\la}\Big)^{\!d}\,
 \frac{|f(\psi^{\la_{n+1}}_x)|}{\la_{n+1}^\al} 
 <   \frac{|f(\psi^{\la_{n+1}}_x)|}{\la_{n+1}^\al}, 
 \qquad \text{where} \;\; \psi(y) = \ffi\Big(\frac{\la_{n+1}}{\la} y\Big).
 }
 Note that since $\frac12 < \frac{\la_{n+1}}{\la}\lt 1$, 
 \aa{
 \supp \psi \sub B_2(0) \sub [-\pi,\pi]^d \quad \text{and} \quad
 \|\psi\|_\infty = \|\ffi\|_\infty \lt 1.
 }
Furthermore, $\psi$ annihilates monomials of order up to
$[\al]$, as $\ffi$ does. 
 The evaluation of the first term in \rf{eqn-norm-a,x-1}
 (for $x=0$) is straightforward.
 Thus, we obtained \rf{ls-1111} for $\al\gt 0$.
\end{proof}
The following corollary follows directly from the proof of
Lemma \ref{h-e-1111}.
\begin{corollary}
\lb{cr-he}
It holds that for every $n\in\Nnu$
\aa{
 \ms Z^n(\Tr^d) \sub \td{\ms Z}^n(\Tr^d) \quad 
\text{and} \quad |\cdot|_{\td{\ms Z}^n} \ls |\cdot|_{\ms Z^n}.
}
\end{corollary}
\begin{lemma}%{lem}
\lb{lem22-22}
Suppose $\eta\in \ms Z^{-\kp}$, $\kp>0$. Then,
\aa{
|\eta^\eps|_{\ms Z^{-\kp}} \lt |\eta|_{\ms Z^{-\kp}}, \quad \text{where} \quad \eta^\eps: =  \rho_\eps \ast \eta.
}
\end{lemma}%{lem}
\begin{proof}
For the function 
$(\eta^{\bar \eps})^\eps: = \rho_\eps \ast \eta^{\bar \eps}$, we have
\mm{
(\eta^{\bar \eps})^\eps(\ffi^\la_x) = \int  \eta^{\bar \eps}(z-y)
\rho_\eps(y) \ffi^\la_x(z) dy dz = 
%- \int \eta^{\bar \eps}(y-\fdot)(\ffi^\la_x) \rho_\eps(y) dy = 
\int  \eta^{\bar \eps}(z) \rho_\eps(y) \ffi^\la_x(z+y) dy dz\\ 
= \int \eta^{\bar \eps} (\ffi^\la_{x-y}) \rho_\eps(y) dy.
}
Passing to the limit as $\bar \eps\to 0$, we obtain
\aa{
%&\eta^\eps(y) =  \int \eta(z-y)\rho_\eps(z) dz
\eta^\eps(\ffi^\la_x) =  \int \eta(\ffi^\la_{x-y}) 
\rho_\eps(y) dy.
}
We have
\aa{
\sup_{\substack{x\in\Tr^d,\, \la\in [0,1]_\D,\\ \ffi\in \ms B^{[\kp+1]}}}
|\eta^\eps(\ffi^\la_x)| 
\lt \int  \sup_{\substack{x\in\Tr^d,\, \la\in [0,1]_\D,\\ \ffi\in \ms B^{[\kp+1]}}}|\eta(\ffi^\la_{x-y})|
\rho_\eps(y) dy
\lt \sup_{\substack{x\in\Tr^d,\, \la\in [0,1]_\D,\\ \ffi\in \ms B^{[\kp+1]}}}|\eta(\ffi^\la_x)|,
}
where the latter inequality is true by the periodicity of $\ffi^\la_x$.
% Note that
% \aa{
% \rho_\eps \ast\ffi^\la_x = \int\rho_\eps(\fdot - x - \la y)\ffi(y) dy =
% \big(\rho_{\frac{\eps}{\la}} \ast \ffi\big)_x^\la.
% }
% By the definition of the norm in $\Zs^\al$,
% \mm{
% |\eta^\eps|_\al =
%  \sup_{\substack{x\in\Tor,\, \la\in (0,1],\\ \ffi\in \ms B^r,\, r=[-\al+1]}}
%  \frac{|\int \eta(\rho_\eps(y-\fdot))\ffi^\la_x(y) dy)|}{\la^\al}
%  =\sup_{\substack{x\in\Tor,\, \la\in (0,1],\\ \ffi\in \ms B^r,\, r=[-\al+1]}}
%  \frac{|\eta(\rho_\eps \ast\ffi^\la_x)|}{\la^\al}\\
%  = \sup_{\substack{x\in\Tor,\, \la\in (0,1],\\ \ffi\in \ms B^r,\, r=[-\al+1]}}
%  \frac{|\eta(\rho_{\frac{\eps}{\la}} \ast \ffi)_x^\la)|}{\la^\al}
%  \lt \sup_{\substack{x\in\Tor,\, \la\in (0,1],\\ \ffi\in \ms B^r,\, r=[-\al+1]}}
%  \frac{|\eta(\ffi_x^\la)|}{\la^\al} = |\eta|_\al.
%  }
%  If $\al\gt 0$, the computation is the same with the exception that the supremum is over
%  $\ffi\in \ms B^0_\al$.
\end{proof}
\begin{lemma}%{lem}
\lb{lem-deri-1111}
Let $\kp>0$ and $\eta\in \ms Z^{-\kp}$. Then,
\aa{
|\pl_i \eta|_{\ms Z^{-\kp-1}} \lt |\eta|_{\ms Z^{-\kp}}.
}
\end{lemma}%{lem}
\begin{proof}
We have
\begin{align*}
|\pl_i \eta|_{\ms Z^{-\kp-1}} &=
 \sup_{\substack{x\in\Tr^d,\, \la\in [0,1]_\D,\\ \ffi\in \ms B^r,\, r=[2+\kp]}}\frac{|\pl_i\eta(\ffi^\la_x)|}{\la^{-1-\kp}}
 = \sup_{\substack{x\in\Tr^d,\, \la\in [0,1]_\D,\\ \ffi\in \ms B^r,\, r=[2+\kp]}}\frac{|\eta(\pl_i\ffi^\la_x)|}{\la^{-\kp}}
\\ &\lt \sup_{\substack{x\in\Tr^d,\, \la\in [0,1]_\D,\\ \psi\in \ms B^r,\, r=[1+\kp]}}
 \frac{|\eta(\psi^\la_x)|}{\la^{-\kp}}
 = |\eta|_{\ms Z^{-\kp}}.
 \end{align*}
\end{proof}
\begin{proposition}%{pro}
\lb{eq1111}
Let $\al<0$.
The H\"older-Zygmund spaces $\ms Z^\al(\Tr^d)$ 
can be equivalently defined with the use of any integer
$r>[-\al+1]$, that is, as the space of distributions 
$f\in \mc D'(\Tr^d)$ with the finite norm
\aaa{
\lb{r-norm-D}
|f|_{\ms Z^\al}^{(r)}:=
 \sup_{\substack{x\in\Tr^d,\, \la\in [0,1]_\D,\\ \ffi\in \ms B^r}}\frac{|f(\ffi^\la_x)|}{\la^\al},  \quad \text{where} \;\;  r>[-\al+1],\; r\in\Nnu.
}
Furthermore, the norms $|\cdot|_{\ms Z^\al}^{(r)}$ and $|\cdot|_{\ms Z^\al}$ 
are equivalent; in particular, 
the norms $|\cdot|_\al$ and $|\cdot|^{(r)}_\al$ are equivalent, where
the latter is the restriction of $|\cdot|_{\ms Z^\al}^{(r)}$
to $\Zs^\al(\Tr^d)$.
\end{proposition}%{pro}
\begin{proof}
Let $r>[-\al+1],\; r\in\Nnu$. Define
\aaa{
\lb{r-norm}
|f|_{\td{\ms Z}^\al}^{(r)}:=
 \sup_{\substack{x\in\Tr^d,\, \la\in (0,1],\\ 
 \ffi\in \td{\ms B}^r}}\frac{|f(\ffi^\la_x)|}{\la^\al}.
}
The equivalence of  the norms \rf{r-norm-D} and \rf{r-norm}
follows by exactly the same argument as in the proof
of Lemma \ref{h-e-1111}. Let us prove the equivalence of
$|\cdot|_{\td{\ms Z}^\al}^{(r)}$ and $|\cdot|_{\td{\ms Z}^\al}$.

Since $\td{\ms Z}^\al(\Tr^d) = \ms Z^\al(\Rnu^d)\cap \ms D'(\Tr^d)$,
see Remark \ref{intersection-1111}, by \cite[Lemma 14.13]{FH20}, 
$|\cdot|_{\td{\ms Z}^\al}^{(r)}$ and $|\cdot|_{\td{\ms Z}^\al}$ 
define the same Banach space. 
% The fact that the norm $|f|_{\ms Z^\al}^{(r)}$ 
% defines the same space $\ms Z^\al$ is proved
% in \cite[Proposition B.1]{bcz2023-1}. 
% Let us show the equivalence of the norms $|\cdot|_{\td{\ms Z}^\al}^{(r)}$ 
% and $|\cdot|_{\td{\ms Z}^\al}$. 
Note that
\aa{
|f|_{\td{\ms Z}^\al}^{(r)} \lt |f|_{\td{\ms Z}^\al}, \quad 
f\in \td{\ms Z}^\al,
}
since the supremum in the definition of $|f|_{\td{\ms Z}^\al}$ is over a larger set.
Since the norms define the same Banach space, their equivalence
is a consequence of the open mapping theorem applied to the bounded linear operator
\aa{
J: \, (\td{\ms Z}^\al, |\cdot|_{\td{\ms Z}^\al}) 
\to (\td{\ms Z}^\al, |\cdot|_{\td{\ms Z}^\al}^{(r)}), \quad f\mto f.
}
Since, by Lemma \ref{h-e-1111}, the norms $|\cdot|_{\td{\ms Z}^\al}$
and $|\cdot|_{\ms Z^\al}$ are equivalent, we conclude
the equivalence of the norms $|\cdot|_{\ms Z^\al}$  and
$|\cdot|_{\ms Z^\al}^{(r)}$.
\end{proof}
\begin{lemma}
\lb{fh2222}
For $\al\notin\Nnu$, $\ms Z^\al(\Tr^d) = \C^\al(\Tr^d)$ and 
the norms in $\ms Z^\al(\Tr^d)$ and $\C^\al(\Tr^d)$ are equivalent. 
Moreover, if $n=[\alpha]$, then 
\[
\C^{n+1} \subset  \ms Z^\al \subset \C^{n}.
\]
\end{lemma}%{pro}
\begin{proof}
By Lemma \ref{h-e-1111}, $\ms Z^\al(\Tr^d) = \td{\ms Z}^\al(\Tr^d)$ and
the norms $|\cdot|_{\ms Z^\al}$ and $|\cdot|_{\td{\ms Z}^\al}$
are equivalent.
By Remark \ref{intersection-1111},
 $\td{\ms Z}^\al(\Tr^d) = \ms Z^\al(\Rnu^d)\cap \ms D'(\Tr^d)$,
 where $\ms Z^\al(\Rnu^d)$ is the classical H\"older-Zygmund space, see \cite{bcz2023-1,hairer}.
 Furthermore, by \cite[Proposition 14.15]{FH20},
$\ms Z^\al(\Rnu^d) = \C^\al(\Rnu^d)$. Therefore,
$\td{\ms Z}^\al(\Tr^d) = \C^\al(\Tr^d)$.

Let us show the equivalence of the norms in $\td{\ms Z}^\al(\Tr^d)$
and $\C^\al(\Tr^d)$.
If $\al<0$, then we have nothing to prove since $\C^\al(\Tr^d)$ is 
$\td{\ms Z}^\al(\Tr^d)$ by definition, see 
\cite[Definition 3.7, Section 3]{hairer}.
Let $\al\gt 0$, $\al\notin \Nnu$. By the same argument as
in \rf{taylor-1111}, we obtain
% $n=[\al]$,  and, for a fixed $x\in\Rnu^d$, let $T_{n,x}$ be the Taylor 
% polynomial of degree $n$, based at 0, for the function $f(x+\fdot)$. 
% Since, the test function $\ffi$ in Definition \ref{HZ1111}
% annihilates monomials  up to degree $n$, Taylor's formula implies
%By Definition \ref{HZ1111},
\aa{
 \sup_{\substack{x\in\Tr^d\!, \,\la\in (0,1],\\ \ffi\in \ms B^0_\al}}\frac{|f(\ffi^\la_x)|}{\la^\al}
 % = \sup_{\substack{x\in\Tr^d\!, \,\la\in (0,1],\\ \ffi\in \ms B^0_\al}}
 % \frac{|(f(x+\la \fdot)- T_{n,x}(\la \fdot))(\ffi)|}{\la^\al}
 \ls \sum_{|\beta| = n}\, [D^\beta f]_\al, 
}
where $\beta$ is a multiindex and $D^\beta$ is a partial
derivative of order $|\beta|$. Hence, 
\aa{
|f|_{\td{\ms Z}^\al} \ls \|f\|_{\C^\al}.
}
Finally, by the same argument as in the proof of Proposition \ref{eq1111},
involving the open mapping theorem, 
the norms $|\cdot|_{\td{\ms Z}^\al}$ and $|\cdot|_{\C^\al}$
are equivalent. This implies the first statement of the lemma.
The last statement of the lemma is an immediate consequence
of the first statement.
\end{proof}
\begin{lemma}
\lb{kp1111}
Assume $\kp, \nu >0$, $\nu \notin \Nnu$. Then, $\ms Z^{\nu} 
\sub \ms Z^{-\kp}$ and
\aaa{
\lb{hz-norms-cmp}
|f|_{\ms Z^{-\kp}} \ls |f|_{\ms Z^\nu}.
}
\end{lemma}
\begin{proof}
Take $f\in \ms Z^{\nu}$ and let $\ell = [\kp+1]$. By Lemma \ref{fh2222},
\aa{
|f|_{\ms Z^{-\kp}} =&\,\sup_{\substack{x\in\Tr^d,\, \la\in [0,1]_\D,\\ \ffi\in \ms B^\ell}}\la^\kp \big|f(\ffi^\la_x)\big|
=\sup_{\substack{x\in\Tr^d,\, \la\in [0,1]_\D,\\ \ffi\in \ms B^\ell}}
\la^\kp \big|f(x+\la\fdot)(\ffi)\big|\\
\lt &\, \sup_{x\in\Tr^d,\, \la\in [0,1]_\D}\la^\kp |f(x)|\,\sup_{\ffi\in \ms B^\ell} \Big|
\int_{\Tr^d} \ffi(y) dy\Big|\\
+& \,\sup_{\substack{x\in\Tr^d,\, \la\in [0,1]_\D,\\ \ffi\in \ms B^\ell}} \la^{\kp+\nu} \Big|
\int_{\Tr^d} |y|^\nu \frac{\big|f(x+\la y)- f(x)\big|}{|\la y|^\nu}\ffi(y) dy\Big|\\
\ls &\, \big[\|f\|_\infty+ [f]_\nu\big] \,\sup_{\ffi\in \ms B^\ell} \Big|\int_{\Tr^d} \ffi(y) dy\Big| \ls
\|f\|_{\C^{\nu}} \ls |f|_{\ms Z^\nu}.
}
\end{proof}
The following corollary is an immediate consequence of 
Lemma \ref{kp1111}.
\begin{corollary} 
\lb{kp2222}
Assume $\kp, \nu >0$, $\nu \notin \Nnu$. Then, $\Zs^\nu \sub \Zs^{-\kp}$ and
\aaa{
\lb{little-hz-norms-cmp}
|f|_{-\kp} \ls |f|_{\nu}.
}
\end{corollary}
\begin{proof}
Inequality \rf{hz-norms-cmp} immediately implies that $\Zs^\nu \sub \Zs^{-\kp}$
and \rf{little-hz-norms-cmp} holds.
\end{proof}
\begin{lemma}
\lb{smooth}
The following equality is true:
 \aa{
 \C^\infty = \medcap_{\al\in\Rnu} \, \ms Z^\al = \medcap_{\al\in\Rnu} \, \Zs^\al.
 }
\end{lemma}
\begin{proof} 
By Lemma \ref{fh2222} and since $\C^\infty\sub \ms Z^n$, $n\in \Nnu$,
it follows that $\C^\infty \sub \medcap_{\al\in\Rnu} \, \ms Z^\al$.
On the other hand, by Lemmata \ref{kp1111} and \ref{fh2222},
$\medcap_{\al\in\Rnu} \, \ms Z^\al
\sub \medcap_{\al\notin \Nnu, \al>0} \, \C^\al 
\sub \medcap_{\al\notin \Nnu, \al>0} \, \C^{[\al]} =
\C^\infty$. It remains to notice that
$\medcap_{\al\in\Rnu} \, \Zs^\al \subset \medcap_{\al\in\Rnu} \, \ms Z^\al$
and that $\C^\infty \subset \Zs^\al$ for all $\al\in\Rnu$.
\end{proof}
\begin{lemma}
\label{lem:12.4}
Given a distribution $f \in \mathcal{D}'(\Tr^d)$ and $\alpha \in (-\infty, 0]$, the following conditions are equivalent.
\begin{enumerate}
    \item $f$ is in $\mc Z^\alpha(\Tr^d)$.
    %    \item There is an integer $r > -\alpha$ such that (12.1) holds with $\mathfrak{B}_{r_\alpha}$ replaced by $\mathfrak{B}_r$.
    \item There exists a test function $\varphi \in \mathcal{D}(\Tr^d)$ with $\int \varphi \neq 0$ such that
    \aa{
        |f(\varphi_x^\la)| \lesssim \la^\alpha
} uniformly in $x\in\Tr^d$  and in $\la \in [0,1]_\D$.
%$\varepsilon \in \{2^{-k}\}_{k \in \mathbb{N}} \subseteq (0, 1]$.
\end{enumerate}
Moreover, there is a constant $\mathfrak{b}_{\varphi,\alpha,d}>0$
such that the H\"older-Zygmund norm $\|f\|_{\mc Z^\alpha}$ 
can be estimated using an arbitrary 
test function $\varphi \in \mathcal{D}(B_1(0))$ with $\int \varphi \neq 0\!:$ 
\begin{equation}
\label{eq:12.3}
    \|f\|_{\mc Z^\alpha} \lt \mathfrak{b}_{\varphi,\alpha,d} \sup_{x \in \Tr^d, \la \in (0,1]} \frac{|f(\varphi_x^\la)|}{\la^\alpha}.
\end{equation}
%where $\mathfrak{b}_{\varphi,\alpha,d}>0$ is a constant. 
\end{lemma}
\begin{proof}
For the proof, see \cite[Theorem 12.4]{cz2020}, whose formulation
is adapted to our spatially periodic setting.
\end{proof}
\section{Singular kernels}
\begin{lemma}%{pro}
\lb{fh1111}
Let $\beta>0$.
For a $\beta$-regularizing kernel $\mc K$ 
{\rm (}see \cite[Definition 2.3]{bcz2023-1}
or \cite[Assumption 5.1]{hairer} {\rm )},
the map $f\mto \mc K\ast f$ is continuous from 
$\ms Z^\al(\Rnu^d)$ to $\ms Z^{\al+\beta}(\Rnu^d)$
for every $\al\in\Rnu$. Moreover, 
\aaa{
\lb{tdK-1111}
|\mc K\ast f|_{\beta+\al, \mcc K} \ls |f|_{\al, \td {\mcc K}},
}
where $\mcc K\sub\Rnu^d$ is a compact and $\td {\mcc K}
\supset \mcc K$ is another compact strictly containing $\mcc K$.
 %where $\Zs^\al$ is the H\"older-Zygmund space.
\end{lemma}%{pro}
\begin{proof}
For the proof see, e.g.  \cite[Theorem 14.17]{FH20}.
The bound \rf{tdK-1111} is obtained in the proof of Theorem 2.13 in 
\cite{bcz2023-1}. 
\end{proof}
In the periodic case, the above statement implies the following result.
\begin{lemma}
\lb{rk1111}
Let $\beta>0$ and
let  $\mc K$ be a $\beta$-regularizing kernel on a $d$-dimensional torus $\Tr^d$.
Then, for every $\al\in\Rnu$, the map 
\aaa{
\lb{map1}
\ms Z^\al(\Tr^d) \to \ms Z^{\al+\beta}(\Tr^d), \quad f\mto \mc K\ast f
}
is linear and continuous. Moreover,  for every $\al\in\Rnu$, the map \rf{map1} maps
$\Zs^\al(\Tr^d)$ to $\Zs^{\al+\beta}(\Tr^d)$. In particular,
\aa{
|\mc K\ast f|_{\ms Z^{\beta+\al}} \ls |f|_{\ms Z^\al}
\quad \text{and} \quad |\mc K\ast f|_{\beta+\al} \ls |f|_{\al}.
}
\end{lemma}
\begin{proof}
By Lemma \ref{h-e-1111}, the norms $|\cdot|_{\ms Z^\al}$ (in the sense of Definition \ref{HZ1111}) 
and $|\cdot|_{\al,\Tor}$ (in the sense of \cite[Definition 2.1]{bcz2023-1})
are equivalent. We remark that $|\cdot|_{\al,\Tr^d}$ 
and $|\cdot|_{\td{\ms Z}^\al}$
is the same norm, see Definition \ref{C-3} and Remark \ref{c5}. 
On the other hand, in Definition \ref{HZ1111},
$\ffi^\la_x$ is a periodic function on $\Tr^d$.
Therefore, the supremum in $x$ running over an enlargement of $\Tr^d$ equals to
the supremum over $\Tr^d$. 

Finally, since $\mc K\ast f  \in C^\infty$  for every $f  \in C^\infty$, 
we infer that  \rf{map1} maps
$\Zs^\al(\Tr^d)$ to $\Zs^{\al+\beta}(\Tr^d)$.
\end{proof}
\begin{lemma}
\lb{Gconv-1111}
The Green function $G$ for the operator
$-2\pl_{\bar z}$ is a $1$-regularizing kernel on $\Tor$. Moreover,
for every $\al\in\Rnu$, the map 
\aaa{
\lb{eqn-conv-alpha}
\ms Z^\al(\Tor) \ni f\mto G\ast f \in  \ms Z^{\al+1}(\Tor),
}
 is linear and continuous,
and maps $\Zs^\al(\Tor)$  to $\Zs^{\al+1}(\Tor)$. 
In particular,
\aaa{
\lb{Gcnv}
|G\ast f|_{\ms Z^{\al+1}} \ls |f|_{\ms Z^\al} \quad \text{and} \quad
|G\ast f|_{\al+1} \ls |f|_\al.
}
\end{lemma}
For the proof, we need the following result. 
\begin{lemma}
\lb{K-bounds-2222}
The Green function $K$ for the Laplacian $-\lap$ on $\Tor$
 is in $\C^\infty(\Tor\setminus\{0\})$. 
%Its behaviour in 0 is given by |G(x)|  C(? log |x| + 1)
For the derivatives of $D^\beta K$, where $\beta$ is a multiindex,
 we have the bounds
\aa{
|D^\beta K(x)| \lt C_n |x|^{-n}, \quad
|\beta|=n, \quad n=1,2,\ldots,
}
where $C_n>0$ is a constant depending only on $n$.
\end{lemma}
\begin{proof}
For the proof see \cite[Proposition B.1]{bfm2016}.
\end{proof}
\begin{proof}[Proof of Lemma \ref{Gconv-1111}]
Since $G=2\pl_z K$, by Lemma \ref{K-bounds-2222},
\aa{
|D^\beta G(x)| \lt C_{n+1} |x|^{-(n+1)}, \quad
|\beta|=n, \quad n=0,1,2,\ldots.
}
In particular, 
\aa{
|G(x)| \lt \frac{C_1}{|x|}. 
}
By \cite[Lemma 2.9]{bcz2023-1}, $G(x-y)$  
is a $1$-regularizing kernel on $\Tor$.
Therefore, the statement of the lemma follows
from Lemma \ref{rk1111}.
% By Lemma \ref{rk1111}, the map \rf{eqn-conv-alpha}
%  %the map $\ms Z^\al(\Tor) \to \ms Z^{\al+1}(\Tor)$, $f\mto G\ast f$, 
% is continuous for all $\al\in\Rnu$ and 
% maps $\Zs^\al(\Tor)$  to $\Zs^{\al+1}(\Tor)$. 
% By the same lemma, \rf{Gcnv} is fulfilled. 
\end{proof}
\begin{lemma}
\lb{laplace-1111}
Let $K$ be the Green function of the Laplacian $-\lap$ on $\Tor$. 
Then, for every $\al\in\Rnu$, the map 
\aaa{
\label{map2}
\ms Z^\al(\Tor) \to \ms Z^{\al+2}(\Tor), \quad 
f\mto K\ast f
}
is linear and continuous. Moreover, for every $\al\in\Rnu$,  the map \rf{map2} maps 
$\Zs^\al(\Tor)$ to $\Zs^{\al+2}(\Tor)$. In particular,
\aa{
|K\ast f|_{\ms Z^{\al+2}} \ls |f|_{\ms Z^\al} \quad \text{and} \quad
|K\ast f|_{\al+2} \ls |f|_\al.
}
\end{lemma}
\begin{proof}
Since $G=2\pl_z K$, we have  
\aa{
G\ast f = 2\pl_z (K\ast f).
}
Since $\int_\Tor G\ast f = 0$, we have
\aa{
K \ast f = - \ovl G \ast (G\ast f).
}
By Lemma \ref{Gconv-1111}, 
\aa{
|K\ast f|_{\ms Z^{\al+2}} = |\ovl G \ast (G\ast f)|_{\ms Z^{\al+2}}
\ls |(G\ast f)|_{\ms Z^{\al+1}} \ls  |f|_{\ms Z^\al}.
}
Finally, we notice that the map \rf{map2} maps 
$\Zs^\al(\Tor)$ to $\Zs^{\al+2}(\Tor)$
by the same argument as in Lemma \ref{rk1111}.
\end{proof}
\section{Young products in  little H\"older-Zygmund spaces}
\begin{lemma}
\lb{young}
Fix  $\al > 0$ and $\beta\lt 0$ such that $\al+\beta > 0$. Then, there exists a unique bilinear continuous map 
\begin{equation}\label{eqn-multiplication}
    \Zs^\al(\Tr^d) \x \Zs^\beta(\Tr^d) \ni  (f,g)\mto  f\cdot g \in  \Zs^\beta(\Tr^d)
    \end{equation}
which extends the usual product  
\aaa{
\lb{usual-1111}
\C^\infty(\Tr^d) \times \C^\infty(\Tr^d)  \ni  (f,g)\mto  f\cdot g \in \C^\infty(\Tr^d).
}
Furthermore, for all $f\in \Zs^\al(\Tr^d)$ and 
$g\in \Zs^\beta(\Tr^d)$,
\aaa{
\lb{Y-torus-1111}
|f\cdot g|_{\beta} \ls |f|_{\al} |g|_{\beta}.
}
\end{lemma}%{lem}
\begin{proof}
By the definition of the spaces $\Zs^\al(\Tr^d)$  and   $\Zs^\beta(\Tr^d)$ 
it is sufficient to prove that \rf{Y-torus-1111} is fulfilled
for all $f,g \in \C^\infty(\Tr^d)$. Then, the product
\rf{usual-1111} will be automatically extended to the map
\rf{eqn-multiplication}, and moreover, \rf{Y-torus-1111} will be fulfilled.
% \aa{
% |f\cdot g|_{\beta} \ls |f|_{\ms Z^\al} |g|_{\beta}.
% }
% This inequality has been proved  \cite[Theorem 14.1]{cz2020} for all 
% $f \in Z^\al(\Tr^d)$ and $g \in \ms Z^\beta(\Tr^d)$. 

In \cite[Theorem 14.1]{cz2020}, it was shown that 
the inequality
\aaa{
\lb{Y-1111}
|f\cdot g|_{K, \beta} \ls |f|_{K_4, \al} |g|_{K_4, \beta},
}
is valid, in particular, for $f,g\in \C^\infty(\Rnu^d)$,
where $K\sub \Rnu^d$ is a compact and $K_4: = K + \ovl{B_4(0)}$ is its
enlargement with $B_4(0)$ being a ball of radius 4 centered
at the origin, see \cite[Formula (14.5)]{cz2020}.
In the periodic case, we take $K=\Tr^d$ and note
that the norms $|\cdot|_{\ms Z^\beta}$ and $|\cdot|_{\Tr^d,\beta}$,
and also $|\cdot|_{\ms Z^\al}$ and $|\cdot|_{\Tr^d,\al}$, $\al\notin \Nnu$, are equivalent
by Lemma \ref{h-e-1111} and Remark \ref{c5}.
 Then, \rf{Y-1111} implies \rf{Y-torus-1111} for $\C^\infty(\Tr^d)$-functions
%  \aa{
% |f\cdot g|_{\ms Z^\beta} \ls |f|_{\ms Z^\al} |g|_{\ms Z^\beta}
% }
by the same argument as in the proof of Lemma \ref{rk1111} for $\al\notin \Nnu$.
If $\al\in \Nnu$, we obtain
\rf{Y-torus-1111}  by Corollary \ref{cr-he}.
% It remains to notice that $\C^\infty(\Tr^d)$ is dense in $\Zs^\beta(\Tr^d)$
% and $\Zs^\al(\Tr^d)$ with respect to the norms of these spaces.
% This implies that the extension of the map 
% \rf{usual-1111} to the map \rf{eqn-multiplication} is unique.
\end{proof}
\begin{remark}
\rm
In \cite[Theorem 14.1]{cz2020}, has was proved that the Young product
\begin{equation*}
    \ms Z^\al(\Tr^d) \x \ms Z^\beta(\Tr^d) \to \ms Z^\beta(\Tr^d), \quad 
     (f,g)\mto  f\cdot g 
        \end{equation*}
        exists and extends the product \rf{usual-1111}, but it is not unique,
        see \cite[Remark 14.2]{cz2020}.
        By Lemma \ref{young}, the extension of the product \rf{usual-1111}
        to $\Zs^\al(\Tr^d) \x \Zs^\beta(\Tr^d)$ is unique and 
        $\Zs^\beta(\Tr^d)$-valued.
\end{remark}
\subsection*{Acknowledgements}
We express our gratitude for the support received during the preparation of this work. Z.B.\ acknowledges partial support from CNPq-Brazil (grant 409764/2023-0). E.S.\ acknowledges partial support from Universidade Federal da Para\'iba (PROPESQ/PRPG/UFPB) through grant PIA13631-2020 (public calls nos.\ 03/2020 and 06/2021), as well as from CNPq-Brazil (grants 409764/2023-0 and 443594/2023-6) and CAPES MATH AMSUD (grant 88887.878894/2023-00). 
  M.N., who is a member of the CNPq-Brazil (grant 443594/2023-6) and CAPES MATH AMSUD grant teams, acknowledges the collaborative support of these projects. 
 He also  acknowledges partial support from Universidade Federal do Amazonas
and 
Coordena\c c\~ao de Aperfei\c coamento de Pessoal de N\'ivel Superior (CAPES). 
Finally, E.S.\ extends thanks to the Mathematics Department of the 
University of York (UK) for their hospitality, as this paper was partially written during her visit to the department.

\end{document}